\documentclass[twoside,11pt,reqno]{amsart}
\usepackage{amsmath,amssymb,amscd,mathrsfs,epic,wasysym,latexsym,tikz,mathrsfs,cite,hyperref}
\usepackage{pb-diagram}
\usepackage[matrix,arrow]{xy}

\usepackage{ stmaryrd }

\usepackage[enableskew]{youngtab}
\usepackage{ytableau}
\usepackage{color}
\usepackage{enumerate}
\usepackage[margin=1.1in]{geometry}

\usetikzlibrary{shapes,snakes,calendar,matrix,backgrounds,folding}

\usepackage{young}

\makeatletter

\usepackage{tikz}
\usetikzlibrary{decorations.pathreplacing,decorations.pathmorphing,calc}

\numberwithin{equation}{section}

\newtheorem{Lemma}[equation]{Lemma}
\newtheorem{Theorem}[equation]{Theorem}
\newtheorem{Corollary}[equation]{Corollary}
\theoremstyle{definition}  
\newtheorem{Definition}[equation]{Definition}
\newtheorem{Remark}[equation]{Remark}
\newtheorem{Algorithm}[equation]{Algorithm}

\newtheorem{Example}[equation]{Example}

\newtheorem{Conjecture}[equation]{Conjecture}

\newtheorem*{thmA}{Theorem A}
\newtheorem*{thmB}{Theorem B}

\let\<\langle
\let\>\rangle

\newcommand\Comment[2][\relax]{\space\par\medskip\noindent%
   \fbox{\begin{minipage}{\textwidth}\textbf{Comment\ifx\relax#1\else---#1\fi}\newline%
        #2\end{minipage}}\medskip
}

\newcommand{\hackcenter}[1]{
 \xy (0,0)*{#1}; \endxy}

\def\b1{\text{\boldmath$1$}}

\def\pmod#1{\text{ }(\text{\rm mod } #1)\,}

\newcommand{\Z}{\mathbb{Z}}

\newcommand{\M}{\textsf{Q}}
\newcommand{\m}{\textsf{q}}

\def\phi{{\varphi}}

\newcommand{\ZZ}{{\mathbb Z}}
\newcommand{\NN}{{\mathbb N}}

\newcommand{\bu}{\boldsymbol{u}}
\newcommand{\bv}{\boldsymbol{v}}
\newcommand{\bw}{\boldsymbol{w}}

\def\b{\mathfrak{b}}
\def\k{\Bbbk}

\theoremstyle{remark}

\newtheorem{ppp}{Problem}

\newtheorem*{pppp}{Problem 1, \(k=2\)}

{\catcode`\|=\active
  \gdef\set#1{\mathinner{\lbrace\,{\mathcode`\|"8000%
  \let|\midvert #1}\,\rbrace}}
}
\def\midvert{\egroup\mid\bgroup}

\colorlet{darkgreen}{green!50!black}
\tikzset{dots/.style={very thick,loosely dotted},
         greendot/.style={fill,circle,color=darkgreen,inner sep=1.5pt,outer sep=0}
}
\def\greendot(#1,#2){\node[greendot] at(#1,#2){}}

\newenvironment{braid}{
  \begin{tikzpicture}[baseline=6mm,blue,line width=1pt, scale=0.4,
                      draw/.append style={rounded corners},
                      every node/.append style={font=\fontsize{5}{5}\selectfont}]%
  }{\end{tikzpicture}
}

\def\Grid(#1,#2){
  \draw[very thin,gray,step=2mm] (0,0)grid(#1,#2);
  \draw[very thin,darkgreen,step=10mm] (0,0)grid(#1,#2);
}

\newcommand\Tableau[2][\relax]{
  \begin{tikzpicture}[scale=0.5,draw/.append style={thick,black}]
    \ifx\relax#1\relax%
    \else 
      \foreach\box in {#1} { \filldraw[blue!30]\box+(-.5,-.5)rectangle++(.5,.5); }
    \fi
    \newcount\row\newcount\col
    \row=0
    \foreach \Row in {#2} {
       \col=1
       \foreach\k in \Row {
          \draw(\the\col,\the\row)+(-.5,-.5)rectangle++(.5,.5);
          \draw(\the\col,\the\row)node{\k};
          \global\advance\col by 1
       }
       \global\advance\row by -1
    }
  \end{tikzpicture}
}

\newcommand\YoungDiagram[2][\relax]{
  \begin{tikzpicture}[scale=0.5,draw/.append style={thick,black}]
    \ifx\relax#1\relax%
    \else 
    \foreach\box in {#1} {
      \filldraw[blue!30]\box rectangle ++(1,1);
    }
    \fi
    \newcount\row
    \row=0
    \foreach \col in {#2} {
       \draw(1,\the\row)grid ++(\col,1);
       \global\advance\row by -1
    }
  \end{tikzpicture}
}

\begin{document}


\title[Searching for quicksand ideals in partially ordered sets]{{\bf Searching for quicksand ideals in partially ordered sets}}

\author{\sc Alexas Iams}
\address{Washington \& Jefferson College\\ Washington\\ PA~15301, USA}
\email{iamsal@washjeff.edu}

\author{\sc Hannah Johnston}
\address{Washington \& Jefferson College\\ Washington\\ PA~15301, USA}
\email{johnstonhm@washjeff.edu}

\author{\sc Robert Muth}
\address{Department of Mathematics\\ Washington \& Jefferson College\\ Washington\\ PA~15301, USA}
\email{rmuth@washjeff.edu}



\begin{abstract}

We consider a combinatorial question about searching for an unknown ideal \(\mu\) within a known poset \(\lambda\). Elements of \(\lambda\) may be queried for membership in \(\mu\), but at most \(k\) positive query results are permitted. The goal is to find a search strategy which guarantees a solution in a minimal total number \(\m_k(\lambda)\) of queries. We provide tight bounds for \(\m_k(\lambda)\), and construct optimal search strategies for the case where \(k=2\) and \(\lambda\) is the product poset of totally ordered finite sets, one of which has cardinality not more than six.
\end{abstract}

\maketitle

\section{Introduction}

\subsection{Quicksand puzzle}\label{minepuz}
A surveyor stands in the northeast corner of a rectangular field \(\lambda\) of dimension $m \times n$. In the southwest corner of the field there may exist a rectangular quicksand pit \(\mu\) of unknown dimension $m' \times n'$. The surveyor has \(k\) stones available to toss into the field in order to identify safe and unsafe regions of the field.
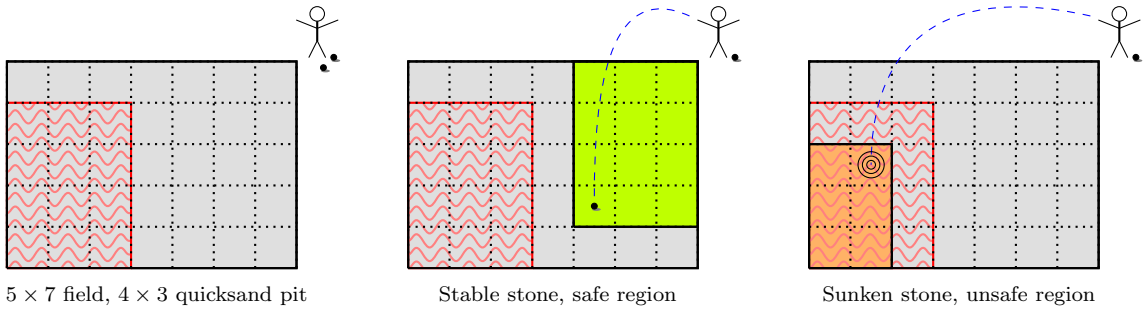
\begin{align*}
\begin{array}{ccccc}
\hackcenter{
\begin{tikzpicture}[scale=0.55]
  \draw[white] (6.8, 6.1).. controls ++(-4,0.9) and ++(0,2.5) .. (1.5,2.5);
\fill[fill=white] (0,0)--(-1,0);
%
\draw[thick, fill=lightgray!50]  (0,0)--(7,0)--(7,5)--(0,5)--(0,0);
\draw[thick, decorate, decoration={snake}, red!50] (0,0)--(3,0);
\draw[thick, decorate, decoration={snake}, red!50] (0,0.33)--(3,0.33);
\draw[thick, decorate, decoration={snake}, red!50] (0,0.66)--(3,0.66);
\draw[thick, decorate, decoration={snake}, red!50] (0,1)--(3,1);
\draw[thick, decorate, decoration={snake}, red!50] (0,1.33)--(3,1.33);
\draw[thick, decorate, decoration={snake}, red!50] (0,1.66)--(3,1.66);
\draw[thick, decorate, decoration={snake}, red!50] (0,2)--(3,2);
\draw[thick, decorate, decoration={snake}, red!50] (0,2.33)--(3,2.33);
\draw[thick, decorate, decoration={snake}, red!50] (0,2.66)--(3,2.66);
\draw[thick, decorate, decoration={snake}, red!50] (0,3)--(3,3);
\draw[thick, decorate, decoration={snake}, red!50] (0,3.33)--(3,3.33);
\draw[thick, decorate, decoration={snake}, red!50] (0,3.66)--(3,3.66);
\draw[thick, decorate, decoration={snake}, red!50] (0,4)--(3,4);
\fill[fill=lightgray!50] (0,4)--(3,4)--(3,4.4)--(0,4.4)--(0,0);
\fill[fill=white] (0,0)--(3,0)--(3,-0.2)--(0,-0.2)--(0,0);
\draw[thick, red] (0,0)--(3,0)--(3,4)--(0,4)--(0,0);
\draw[thick, dotted] (0,0)--(0,5);
\draw[thick, dotted] (1,0)--(1,5);
\draw[thick, dotted] (2,0)--(2,5);
\draw[thick, dotted] (3,0)--(3,5);
\draw[thick, dotted] (4,0)--(4,5);
\draw[thick, dotted] (5,0)--(5,5);
\draw[thick, dotted] (6,0)--(6,5);
\draw[thick, dotted] (7,0)--(7,5);
\draw[thick, dotted] (0,1)--(7,1);
\draw[thick, dotted] (0,2)--(7,2);
\draw[thick, dotted] (0,3)--(7,3);
\draw[thick, dotted] (0,4)--(7,4);
\draw[thick, dotted] (0,5)--(7,5);
\draw[thick]  (0,0)--(7,0)--(7,5)--(0,5)--(0,0);
%
%
%
\draw[thick, line cap=round] (7.5,5.5)--(7.5,6); 
\draw[line cap=round] (7.5,5.5)--(7.3, 5.1); 
\draw[line cap=round] (7.5,5.5)--(7.7,5.1); 
\draw[line cap=round] (7.5,5.8)--(7,6); 
\draw[line cap=round] (7.5,5.8)--(8, 6); 
\draw[fill=white]  (7.5,6.15) circle (5pt);
 \fill[fill=gray]  (7.95,5.02) ellipse (3pt and 1pt);
\draw[fill=black]  (7.9,5.1) circle (2pt);
 \fill[fill=gray]  (7.95 -0.25,5.02-0.25) ellipse (3pt and 1pt);
\draw[fill=black]  (7.9-0.25,5.1-0.25) circle (2pt);
%
%
\end{tikzpicture}
}
&
\hackcenter{
\begin{tikzpicture}[scale=0.55]
\fill[fill=white] (0,0)--(-1,0);
%
  \draw[white] (6.8, 6.1).. controls ++(-4,0.9) and ++(0,2.5) .. (1.5,2.5);
\draw[thick, fill=lightgray!50]  (0,0)--(7,0)--(7,5)--(0,5)--(0,0);
\draw[thick, fill=lime]  (4,1)--(7,1)--(7,5)--(4,5)--(4,1);
\draw[thick, decorate, decoration={snake}, red!50] (0,0)--(3,0);
\draw[thick, decorate, decoration={snake}, red!50] (0,0.33)--(3,0.33);
\draw[thick, decorate, decoration={snake}, red!50] (0,0.66)--(3,0.66);
\draw[thick, decorate, decoration={snake}, red!50] (0,1)--(3,1);
\draw[thick, decorate, decoration={snake}, red!50] (0,1.33)--(3,1.33);
\draw[thick, decorate, decoration={snake}, red!50] (0,1.66)--(3,1.66);
\draw[thick, decorate, decoration={snake}, red!50] (0,2)--(3,2);
\draw[thick, decorate, decoration={snake}, red!50] (0,2.33)--(3,2.33);
\draw[thick, decorate, decoration={snake}, red!50] (0,2.66)--(3,2.66);
\draw[thick, decorate, decoration={snake}, red!50] (0,3)--(3,3);
\draw[thick, decorate, decoration={snake}, red!50] (0,3.33)--(3,3.33);
\draw[thick, decorate, decoration={snake}, red!50] (0,3.66)--(3,3.66);
\draw[thick, decorate, decoration={snake}, red!50] (0,4)--(3,4);
\fill[fill=lightgray!50] (0,4)--(3,4)--(3,4.4)--(0,4.4)--(0,0);
\fill[fill=white] (0,0)--(3,0)--(3,-0.2)--(0,-0.2)--(0,0);
\draw[thick, red] (0,0)--(3,0)--(3,4)--(0,4)--(0,0);
\draw[thick, dotted] (0,0)--(0,5);
\draw[thick, dotted] (1,0)--(1,5);
\draw[thick, dotted] (2,0)--(2,5);
\draw[thick, dotted] (3,0)--(3,5);
\draw[thick, dotted] (4,0)--(4,5);
\draw[thick, dotted] (5,0)--(5,5);
\draw[thick, dotted] (6,0)--(6,5);
\draw[thick, dotted] (7,0)--(7,5);
\draw[thick, dotted] (0,1)--(7,1);
\draw[thick, dotted] (0,2)--(7,2);
\draw[thick, dotted] (0,3)--(7,3);
\draw[thick, dotted] (0,4)--(7,4);
\draw[thick, dotted] (0,5)--(7,5);
\draw[thick]  (0,0)--(7,0)--(7,5)--(0,5)--(0,0);
\draw[thick, line cap=round] (7.5,5.5)--(7.5,6); 
\draw[line cap=round] (7.5,5.5)--(7.3, 5.1); 
\draw[line cap=round] (7.5,5.5)--(7.7,5.1); 
\draw[line cap=round] (7.5,5.8)--(7,6); 
\draw[line cap=round] (7.5,5.8)--(8, 6); 
\draw[fill=white]  (7.5,6.15) circle (5pt);
 \fill[fill=gray]  (7.95,5.02) ellipse (3pt and 1pt);
\draw[fill=black]  (7.9,5.1) circle (2pt);
%
%
  \draw[dashed, blue] (6.8, 6.1).. controls ++(-2,0.9) and ++(0,2.5) .. (4.5,1.7);
 %
 %
 %
  \fill[fill=gray]  (4.5+0.05,1.5-0.08) ellipse (3pt and 1pt);
\draw[fill=black]  (4.5,1.5) circle (2pt);
\draw[thick]  (4,1)--(7,1)--(7,5)--(4,5)--(4,1);
\end{tikzpicture}
}
&
\hackcenter{
\begin{tikzpicture}[scale=0.55]
\fill[fill=white] (0,0)--(-1,0);
%
  \draw[white] (6.8, 6.1).. controls ++(-4,0.9) and ++(0,2.5) .. (1.5,2.5);
\draw[thick, fill=lightgray!50]  (0,0)--(7,0)--(7,5)--(0,5)--(0,0);
\draw[thick, fill=orange!60]  (0,0)--(2,0)--(2,3)--(0,3)--(0,0);
\draw[thick, decorate, decoration={snake}, red!50] (0,0)--(3,0);
\draw[thick, decorate, decoration={snake}, red!50] (0,0.33)--(3,0.33);
\draw[thick, decorate, decoration={snake}, red!50] (0,0.66)--(3,0.66);
\draw[thick, decorate, decoration={snake}, red!50] (0,1)--(3,1);
\draw[thick, decorate, decoration={snake}, red!50] (0,1.33)--(3,1.33);
\draw[thick, decorate, decoration={snake}, red!50] (0,1.66)--(3,1.66);
\draw[thick, decorate, decoration={snake}, red!50] (0,2)--(3,2);
\draw[thick, decorate, decoration={snake}, red!50] (0,2.33)--(3,2.33);
\draw[thick, decorate, decoration={snake}, red!50] (0,2.66)--(3,2.66);
\draw[thick, decorate, decoration={snake}, red!50] (0,3)--(3,3);
\draw[thick, decorate, decoration={snake}, red!50] (0,3.33)--(3,3.33);
\draw[thick, decorate, decoration={snake}, red!50] (0,3.66)--(3,3.66);
\draw[thick, decorate, decoration={snake}, red!50] (0,4)--(3,4);
\fill[fill=lightgray!50] (0,4)--(3,4)--(3,4.4)--(0,4.4)--(0,0);
\fill[fill=white] (0,0)--(3,0)--(3,-0.2)--(0,-0.2)--(0,0);
\draw[thick, red] (0,0)--(3,0)--(3,4)--(0,4)--(0,0);
\draw[thick, dotted] (0,0)--(0,5);
\draw[thick, dotted] (1,0)--(1,5);
\draw[thick, dotted] (2,0)--(2,5);
\draw[thick, dotted] (3,0)--(3,5);
\draw[thick, dotted] (4,0)--(4,5);
\draw[thick, dotted] (5,0)--(5,5);
\draw[thick, dotted] (6,0)--(6,5);
\draw[thick, dotted] (7,0)--(7,5);
\draw[thick, dotted] (0,1)--(7,1);
\draw[thick, dotted] (0,2)--(7,2);
\draw[thick, dotted] (0,3)--(7,3);
\draw[thick, dotted] (0,4)--(7,4);
\draw[thick, dotted] (0,5)--(7,5);
\draw[thick]  (0,0)--(7,0)--(7,5)--(0,5)--(0,0);%
\draw[thick, line cap=round] (7.5,5.5)--(7.5,6); 
\draw[line cap=round] (7.5,5.5)--(7.3, 5.1); 
\draw[line cap=round] (7.5,5.5)--(7.7,5.1); 
\draw[line cap=round] (7.5,5.8)--(7,6); 
\draw[line cap=round] (7.5,5.8)--(8, 6); 
\draw[fill=white]  (7.5,6.15) circle (5pt);
 \fill[fill=gray]  (7.95,5.02) ellipse (3pt and 1pt);
\draw[fill=black]  (7.9,5.1) circle (2pt);
%
%
%
  \draw[dashed, blue] (6.8, 6.1).. controls ++(-4,0.9) and ++(0,2.5) .. (1.5,2.5);
         \draw[black]  (1.5,2.5) circle (3pt);
          \draw[black]  (1.5,2.5) circle (6pt);
          \draw[black]  (1.5,2.5) circle (9pt);
 %
 %
 %
 \draw[thick]  (0,0)--(2,0)--(2,3)--(0,3)--(0,0);
\end{tikzpicture}
}
\\
\scriptstyle{\textup{{ \(5 \times 7\) field, \(4 \times 3\) quicksand pit}}} 
&\scriptstyle{\textup{{ Stable stone, safe region}}}
& \scriptstyle{\textup{{ Sunken stone, unsafe region}}}
\end{array}
\end{align*}
\vspace{0.5mm}

In order to gain information, the surveyor tosses a stone into some location \(x\) in the field. If the stone does not sink, it follows that the region northeast of \(x\) is safe; the surveyor can venture into the field to retrieve the stone and use it again. If the stone does sink, the surveyor knows that the quicksand pit extends at least as far as \(x\), but they now have one less stone with which to work. How can the surveyor identify the location of the quicksand pit, and do so in a minimal number of tosses?

\subsection{Quicksand ideals in posets}\label{QP}
As we explain in \S\ref{NN}, this puzzle is a special case of a more general problem.
Let \(\lambda\) be a finite poset and \(k \in \NN\). We seek to identify a (possibly empty) `quicksand' ideal \(\mu\) contained in \(\lambda\) by sequentially querying elements of \(\lambda\) for membership in \(\mu\), under the restriction that at most \(k\) positive query results are permitted. Letting \(\m_k(\lambda)\) represent the minimum total number of queries needed to guarantee identification of \(\mu\), our goal is to solve:

\begin{ppp}\label{P1}
Find the value \(\m_k(\lambda)\), and identify a search strategy which realizes this value.
\end{ppp}

For all \(k \in \NN\), the value \(\m_k(\lambda)\) has a recursive combinatorial description, as explained in \S\ref{qkexp}:
\begin{align*}
\m_k(\lambda) = 
\begin{cases}
0 &\textup{if }\lambda = \varnothing;\\
|\lambda| &\textup{if }k = 1;\\
\min \{ \max\{ \m_k(\lambda_{\not \succeq u}) , \m_{k-1}(\lambda_{\succ u})\}\mid u \in \lambda\}+1
&  \textup{if }k >1, \lambda \neq \varnothing,
\end{cases}
\end{align*}

For any \(x \in \ZZ_{\geq 0}\), let
\(
T_k(x) = \sum_{i=1}^k {x \choose k},
\)
and let \(\tau_k(x)\) be the smallest integer such that \(x \leq T_k(\tau_k(x))\). Our first main result provides bounds for \(\m_k(\lambda)\):
\begin{thmA}
For all \(k \in \NN\) and posets \(\lambda\), we have \(\tau_k(|\lambda|) \leq \m_k(\lambda) \leq |\lambda|\).
\end{thmA}
This appears as Theorem~\ref{BOUNDS} in the text.
These bounds are tight, in that \(\m_k(\lambda) = |\lambda|\) when \(\lambda\) has the trivial partial order, and \(\m_k(\lambda) = \tau_k(|\lambda|)\) when \(\lambda\) is totally ordered. In fact, when \(\lambda\) is totally ordered, Problem~\ref{P1} is related to the `\(k\)-egg' or `\(k\)-marble' problem \cite{Egg1, Egg2, Egg3, marbles}, which appears in numerous texts on dynamic programming and optimization, and perhaps apocryphally, as an interview question for certain coding positions in big tech.

\subsection{Quicksand ideals in the product order, \(k=2\) case}\label{NN}
After investigating general results described in \S\ref{QP}, we devote our attention to a special case of Problem 1. 
When \(\kappa, \nu\) are totally ordered sets, we consider \(\kappa \times \nu\) to be a poset under the {\em product partial order}; i.e.,
\begin{align*}
(x_1, y_1) \succeq (x_2, y_2) \iff x_1 \geq x_2 \textup{ and } y_1 \geq y_2,
\end{align*}
for \(x_1, x_2 \in \kappa\) and \(y_1, y_2 \in \nu\).
We consider the \(k=2\) case, where \(T_2(x)\) is the triangular number \(1 + 2 + \cdots + x = x(x+1)/2\), and \(\tau_2(x) = \lceil (\sqrt{8x + 1} -1)/2 \rceil\). 
Our second main result, which appears as Corollary~\ref{maincor} in the text, provides a partial solution to Problem~\ref{P1} in this setting:

\begin{thmB}
Let \(\kappa, \nu\) be finite totally ordered sets, with \(|\kappa| \leq 6\) or \(|\nu| \leq 6\). Then
\begin{align*}
\m_2(\kappa \times \nu) = \begin{cases}
9 & \textup{if }|\kappa|=|\nu| =6;\\
\tau_2(|\kappa| |\nu|) & \textup{otherwise}.
\end{cases}
\end{align*}
\end{thmB}

In Algorithm~\ref{Alg} we describe an explicit strategy, for any such \(\kappa,\nu\), which realizes the value \(\m_2(\kappa \times \nu)\) above. In general, this strategy---and hence the proof of Theorem B---is rather delicately connected to the congruence class of \(\tau_2(|\kappa| |\nu|)\) modulo \(|\kappa|\) and \(|\nu|\), and relies heavily on some interesting number theoretic facts about triangular numbers proved in \S\ref{TriResults}. We close the paper with a conjectural upper bound on \(\m_2(\kappa \times \nu)\) in general, see \S\ref{conjsec}.

\subsection{Solving the quicksand puzzle}
Theorem B offers a solution to the puzzle in \S\ref{minepuz} for the case where \(k=2\) and one dimension of the field is not more than six. Indeed, we may consider the field \(\lambda\) as the poset \([1,m] \times [1,n]\), depicted as a rectangular array of boxes in the first quadrant of the Cartesian plane. The quicksand pit is then an unknown ideal in \(\lambda\), since any ideal \(\mu \subseteq \lambda\) is either empty or equal to \([1,m'] \times [1,n']\) for some \(m' \leq m\), \(n' \leq n\).

Take the \(k=2\), \(\lambda = [1,5] \times [1,7]\) example from \S\ref{minepuz} for instance. Algorithm~\ref{Alg} returns an optimal strategy displayed below. 
\begin{align*}
\begin{array}{ccccc}
\hackcenter{
\begin{tikzpicture}[scale=0.55]
%
\fill[fill=white] (0,0)--(3,0)--(3,-0.2)--(0,-0.2)--(0,0);
\draw[thick, fill=lightgray!50]  (0,0)--(7,0)--(7,5)--(0,5)--(0,0);
\draw[thick, fill=brown!35]  (0,0)--(7,0)--(7,5)--(0,5)--(0,0);
\draw[thick, fill=blue!40!green!45]  (0,1)--(7,1)--(7,5)--(0,5)--(0,1);
\draw[thick, fill=orange!60]  (1,0)--(7,0)--(7,5)--(1,5)--(1,0);
\draw[thick,  fill=violet!35]  (0,3)--(7,3)--(7,5)--(0,5)--(0,3);
\draw[thick, fill=cyan!25]  (2,0)--(7,0)--(7,5)--(2,5)--(2,0);
\draw[thick, fill=lime]  (3,0)--(7,0)--(7,5)--(3,5)--(3,0);
\draw[thick, fill=yellow]  (4,1)--(7,1)--(7,5)--(4,5)--(4,1);
\draw[thick, fill=pink]  (3,3)--(7,3)--(7,5)--(3,5)--(3,3);
%
\draw[thick, dotted] (0,0)--(0,5);
\draw[thick, dotted] (1,0)--(1,5);
\draw[thick, dotted] (2,0)--(2,5);
\draw[thick, dotted] (3,0)--(3,5);
\draw[thick, dotted] (4,0)--(4,5);
\draw[thick, dotted] (5,0)--(5,5);
\draw[thick, dotted] (6,0)--(6,5);
\draw[thick, dotted] (7,0)--(7,5);
\draw[thick, dotted] (0,1)--(7,1);
\draw[thick, dotted] (0,2)--(7,2);
\draw[thick, dotted] (0,3)--(7,3);
\draw[thick, dotted] (0,4)--(7,4);
\draw[thick, dotted] (0,5)--(7,5);
\draw[thick]  (0,0)--(7,0)--(7,5)--(0,5)--(0,0);
\node at (3.5,3.5){ \textcircled{$\scriptstyle{1}$}};
\node at (4.5,1.5){ \textcircled{$\scriptstyle{2}$}};
\node at (3.5,0.5){ \textcircled{$\scriptstyle{3}$}};
\node at (2.5,0.5){ \textcircled{$\scriptstyle{4}$}};
\node at (0.5,3.5){ \textcircled{$\scriptstyle{5}$}};
\node at (1.5,0.5){ \textcircled{$\scriptstyle{6}$}};
\node at (0.5,1.5){ \textcircled{$\scriptstyle{7}$}};
\node at (0.5,0.5){ \textcircled{$\scriptstyle{8}$}};
\end{tikzpicture}
}
\\
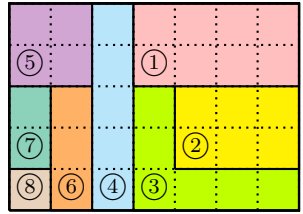
\scriptstyle{\textup{{Strategy for \(5 \times 7\) field}}} 
\end{array}
\end{align*}
The surveyor tosses their first stone into the locations marked \(\textcircled{$\scriptstyle{1}$}, \textcircled{$\scriptstyle{2}$},  \textcircled{$\scriptstyle{3}$}, \ldots, \) in sequence.
If this stone never sinks, then \(\mu = \varnothing\). If the stone sinks on say, the \(i\)th toss, the remaining uncleared area weakly northeast of this location (belonging to the same colored region as \(\textcircled{$\scriptstyle{i}$}\)), is checked sequentially with the remaining stone, in a southwesterly fashion. When the second stone sinks it will determine the northeast corner of \(\mu\), and if it never sinks, the northeast corner of \(\mu\) is at \(\textcircled{$\scriptstyle{i}$}\). This strategy identifies the quicksand pit \(\mu\) in at most \(\tau_2(5\cdot 7) =8\) total tosses.


\section{Partially ordered sets}\label{YDSEC}

In this section we give a brief primer on partially ordered sets and provide some preliminary definitions. See \cite{Dushnik, Davey} for a complete treatment of the subject. We introduce the \(\m_k\)-function which is the central topic of this paper, and explain how it relates to Problem~\ref{P1}.

\subsection{Posets}
A {\em partially ordered set} (or {\em poset}) is a set \(\lambda\) together with a binary relation \(\succeq\), which satisfies the following conditions for all \(u,v,w \in \lambda\):
\begin{enumerate}
\item  \(u \succeq u\) ({\em reflexivity});
\item \(u \succeq v\) and \(v\succeq u\) imply \(u = v\) ({\em antisymmetricity});
\item \(u \succeq v\) and \(v \succeq w\) imply \(u \succeq w\) ({\em transitivity}).
\end{enumerate}
We use \(a \succ b\) to indicate \(a \succeq b\) and \(a \neq b\). The order \(\succeq\) is a {\em total order} if either \(u \succeq v\) or \(v \succeq u\) for all \(u \in v\). An {\em order-preserving} map of posets \(\lambda, \nu\) is a set map \(f: \lambda \to \nu\) such that \(f(u) \succeq f(v)\) whenever \(u \succeq v\). We say two posets \(\lambda, \nu\) are {\em isomorphic} and write \(\lambda \cong \nu\) if there exist mutually inverse order-preserving maps \(\lambda \rightleftarrows \nu\).

If \(\kappa, \nu\) are posets, then \(\kappa \times \nu\) is a poset under the {\em product partial order}:
\begin{align*}
(x_1, y_1) \succeq (x_2, y_2) \iff x_1 \succeq_\kappa x_2 \textup{ and } y_1 \succeq_\nu y_2,
\end{align*}
for all \(x_1, x_2 \in \kappa\) and \(y_1, y_2 \in \nu\). Our main examples of posets in this paper are the following:

\begin{Example}
The {\em trivial partial order} on a set \(\lambda\) has \(u \succeq v\) if and only if \(u = v\) for all \(u,v \in \lambda\). 
\end{Example}

\begin{Example}\label{TotEx}
The natural numbers \(\NN = \{1, 2, \ldots \}\) are totally ordered under the usual \(\geq\) relation, as is any interval \([a,b] = \{a, a+1, \ldots, b\} \subset \NN\). In fact, if \(\lambda\) is any finite totally ordered set of cardinality \(m\), then \(\lambda \cong [1,m]\).
\end{Example}

\begin{Example}
Let \(m,n \in \NN\). Then \([1,m], [1,n]\) are totally ordered sets as in Example~\ref{TotEx}. We write \(\llbracket m,n\rrbracket\) as shorthand for the poset \([1,m] \times [1,n]\) under the product partial order. If \(\kappa, \nu\) are totally ordered sets of cardinality \(m, n\) respectively, then \(\kappa \times \nu \cong \llbracket m,n \rrbracket\).

We represent elements of \(\llbracket m,n \rrbracket\) as boxes situated in the first quadrant of the plane, arranged so that \( (a, b)\) is a box in the \(a\)th row from the bottom, and in the \(b\)th column from the left. In this scheme, we have \(u \succeq v\) for \(u,v \in \llbracket m,n \rrbracket\) if and only if the \(v\) box is weakly below and to the left (i.e. `southwest') of the \(u\) box. For example, in the figure below we show the poset \(\llbracket 5,7 \rrbracket\), with the elements \(x=(4,3), y = (2,5), z=(2,2)\). Then we have \(x \succeq z, y \succeq  z\), with \(x,y\) incomparable.
\begin{align*}
\begin{array}{ccccc}
\hackcenter{
\begin{tikzpicture}[scale=0.55]
%
\fill[fill=white] (0,0)--(3,0)--(3,-0.2)--(0,-0.2)--(0,0);
\draw[thick, fill=lightgray!50]  (0,0)--(7,0)--(7,5)--(0,5)--(0,0);
%
\draw[thick, dotted] (0,0)--(0,5);
\draw[thick, dotted] (1,0)--(1,5);
\draw[thick, dotted] (2,0)--(2,5);
\draw[thick, dotted] (3,0)--(3,5);
\draw[thick, dotted] (4,0)--(4,5);
\draw[thick, dotted] (5,0)--(5,5);
\draw[thick, dotted] (6,0)--(6,5);
\draw[thick, dotted] (7,0)--(7,5);
\draw[thick, dotted] (0,1)--(7,1);
\draw[thick, dotted] (0,2)--(7,2);
\draw[thick, dotted] (0,3)--(7,3);
\draw[thick, dotted] (0,4)--(7,4);
\draw[thick, dotted] (0,5)--(7,5);
\draw[thick]  (0,0)--(7,0)--(7,5)--(0,5)--(0,0);
\node at (2.5,3.5){x};
\node at (4.5,1.5){ y};
\node at (1.5,1.5){z};
\end{tikzpicture}
}
\\
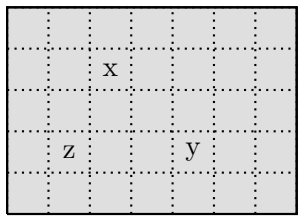
\scriptstyle{\textup{{The poset \(\llbracket 5,7 \rrbracket\), with elements \(x,y,z\)}}} 
\end{array}
\end{align*}
\end{Example}

\subsection{Lower sets and ideals}
Let \(U\) be a subset of a poset \(\lambda\). Then \(U\) is itself a poset under the partial order inherited from \(\lambda\), and we always assume we take this partial order on \(U\). We say \(U\) is a {\em lower set} in \(\lambda\) provided that for all \(u \in U\), \(v \in \lambda\), \(u\succeq v\) implies \(v \in U\). We say \(U\) is a {\em directed set} in \(\lambda\) provided that for all \(u,v \in U\), there exists \(w \in U\) such that \(w \succeq u,v\). We say \(U\) is an {\em ideal} in \(\lambda\) if it is a lower set and a directed set. In particular, we allow ideals to be empty.

Let \(S, U \subseteq \lambda\). We define subsets:
\begin{align*}
S_{ \succeq U} &= \{ v \in S \mid v \succeq u \textup{ for some }u \in U \}\\
S_{ \succ U} &= \{ v \in S \mid v \succ u \textup{ for some }u \in U \} \\
S_{ \preceq U} &= \{ v \in S \mid v \preceq u \textup{ for some }u \in U \}\\
S_{ \not \succeq U} &= \{v \in S \mid v \not \succeq u \textup{ for all }u \in U\} = S \backslash S_{\succeq U}.
\end{align*}
When \(U = \{u\}\), we will write \(S_{\succeq u}\) in place of \(S_{\succeq \{u\}}\), and so on. For any ordered sequence \(\bu = (u_1, \ldots, u_r)\) of elements of \(\lambda\), we will also write  \(S_{\succeq \bu}\) in place of \(S_{\succeq \{u_1, \ldots, u_r\}}\), and so on.  We will often apply these definitions with \(S = \lambda\). 

We will focus primarily on finite posets \(\lambda\). In this setting every ideal is either empty or {\em principal}; i.e. of the form \(\lambda_{\preceq u}\) for some \(u \in \lambda\), and every lower set is equal to \(\lambda_{\preceq U}\) for some \(U\subseteq \lambda\).

\subsection{The \(\m_k\)-function and Problem~\ref{P1}}\label{qkexp}
\begin{Definition}\label{qkdef}
Let \(k \in \NN\), and let \(\lambda\) be a finite poset. We define the value \(\m_k(\lambda) \in \ZZ_{\geq 0}\) recursively by setting:
\begin{align*}
\m_k(\lambda) = 
\begin{cases}
0 &\textup{if }\lambda = \varnothing;\\
|\lambda| &\textup{if }k = 1;\\
\min \{ \max\{ \m_k(\lambda_{\not \succeq u}) , \m_{k-1}(\lambda_{\succ u})\}\mid u \in \lambda\}+1
&  \textup{if }k >1, \lambda \neq \varnothing,
\end{cases}
\end{align*}
where we implicitly take the partial orders on \(\lambda_{\not \succeq u}\), \(\lambda_{\succ u}\) to be those inherited from \(\lambda\).
\end{Definition}

\begin{Example}\label{nonmon}
It is easy to check from Definition~\ref{qkdef} that \(\m_k(\lambda) = |\lambda|\) when \(|\lambda| \leq 2\). Let \(A = \{a,b,c\}\), and consider the posets \(\lambda_0, \lambda_1, \lambda_2, \lambda_3\) with underlying set \(A\), where the strict comparisons in these posets are given as follows:
\begin{align*}
\lambda_0 : \varnothing
\qquad
\lambda_1 : \{b \succ a\}
\qquad
\lambda_2 : \{b,c \succ a\}
\qquad
\lambda_3 : \{c \succ b \succ a\}.
\end{align*}
Note that \(\lambda_0\) is the trivial poset on \(A\) and \(\lambda_3\) is a totally ordered set on \(A\). We have \(\m_1(\lambda_i) = 3\) for \(i=0,1,2,3\) by definition, and it is straightforward to compute:
\begin{align*}
\m_k(\lambda_0) = 3
\qquad
\m_k(\lambda_1) =2
\qquad
\m_k(\lambda_2) = 3
\qquad
\m_k(\lambda_3) = 2
\end{align*}
for all \(k \geq 2\).
\end{Example}

We now explain how this combinatorial function relates to Problem~\ref{P1}. Recall that in Problem~\ref{P1}, \(\mu\) is an unknown ideal in \(\lambda\) we wish to identify, and we may sequentially query elements of \(\lambda\) for membership in \(\mu\), with the restriction that we must stop after the \(k\)th positive query. Note that since \(\mu\) is an ideal in a finite set, we have that \(\mu = \varnothing\) or \(\mu = \lambda_{\preceq x}\) for some \(x \in \lambda\). Let \(\m_k'(\lambda)\) represent the minimum total number of queries needed to guarantee identification of \(\mu\). We explain now that \(\m_k'(\lambda) = \m_k(\lambda)\).

\subsubsection{The \(\lambda = \varnothing\) case} In this case we must have \(\mu = \varnothing\), so no queries are needed to identify \(\mu\). Thus \(\m_k'(\varnothing) =0  = \m_k(\varnothing)\).

\subsubsection{The \(k=1\) case}\label{k1case} With only one positive search query available, the search strategy is very limited. Assume that \(u \in \lambda\) and we know \(v \notin \mu\) for all \(v \succ u\) by previous queries. Then a positive query at \(u\) will identify \(\mu\) to be the ideal \(\lambda_{\preceq u}\). On the other hand, if there exists a element \(v \succ u\) whose membership in \(\mu\) is unknown, a positive query result at \(u\) would result in failure, as \(\mu\) could potentially be \(\lambda_{\preceq v}\) or \(\lambda_{\preceq u}\), and we would be left with no further queries to distinguish these possibilities. 

We see then that the only permissible search strategy is to query all of the elements of \(\lambda\) in some non-increasing sequence, where the first positive query result will identify the generator of the ideal \(\mu\). If \(\mu = \varnothing\), the ideal will only be identified after the final (negative) query, so we have \(\m'_k(\lambda) = |\lambda| = \m_k(\lambda)\). 

\subsubsection{The general \(k>1\), \(|\lambda| > 0\) case} By induction, assume that \(\m'_\ell (\nu) = \m_\ell (\nu)\) for all \(\ell < k\) or \(|\nu| < |\lambda|\). Assume the first query is at some element \(u \in \lambda\). If the query is negative, this implies that \(\mu \subseteq \lambda_{\not \succeq u}\), and we still have \(k\) positive queries to work with. By induction, the minimal total number of queries necessary to guarantee identification of \(\mu\) in \(\lambda_{\not \succeq u}\) is \(\m_k(\lambda_{\not \succeq u})\). 

On the other hand, assume the query at \(u \in \lambda\) is positive. This implies that the ideal generator \(x\) could be any element in \(\lambda_{\succeq u}\), and we now have \(k-1\) positive query results remaining. Let \(\mu'\) be the ideal \(\mu \cap \lambda_{\succ u}\) in \(\lambda_{\succ u}\). Then we have \(\mu' = \varnothing\) if and only if \(x = u\), and \(\mu'\) is nonempty if and only if \(x \in \lambda_{\succ u}\) and \(\mu' = (\lambda_{ \succ u})_{\preceq x}\). Therefore, identifying \(\mu\) is equivalent to identifying the ideal \(\mu'\) in \(\lambda_{ \succ u}\). By induction, \(\m_{k-1}(\lambda_{\succ u})\) is the minimal total number of queries necessary to guarantee success in this search.

Therefore if we begin by querying \(u\), the minimal number of queries that will be necessary to guarantee identification of \(\mu\) in \(\lambda\) is \(\m_k(\lambda_{\not \succeq u}) + 1\) if \(u \notin \mu\), and \(\m_{k-1}(\lambda_{\succ u}) + 1\) if \(u \in \mu\). Thus, by first querying \(u\), the minimal number of queries necessary is
\begin{align*}
\max\{ \m_k(\lambda_{\not \succeq u}), \m_{k-1}(\lambda_{\succ u})\} + 1.
\end{align*}
Therefore, taking the minimum over all possible choices of the initial query \(u\), we have that \(\m_k'(\lambda) = \m_k(\lambda)\), as desired.

\section{Binomial sums and triangular numbers}
Bounds for the \(\m_k\)-function will be shown to be directly related to binomial sums, and, in the \(k=2\) case, triangular numbers. In preparation for establishing this fact, we investigate some properties of binomial sums, and triangular numbers in particular.

\subsection{Binomial sums} Throughout this section, we fix \(k \in \NN\). 
\begin{Definition}
Define the function \(T_k: \ZZ_{\geq 0} \to \ZZ_{\geq 0}\) via:
\begin{align*}
T_k(x) = \sum_{i = 1}^k {x \choose i}.
\end{align*}
\end{Definition}
Notably, when \(k=1\) we have \(T_1(x) = x\), and when \(k=2\) we have
\begin{align}\label{exptri}
T_2(x) = \frac{x(x+1)}{2} = 1 + 2 + \cdots + x,
\end{align}
the \(x\)th {\em triangular number}. The following function is key in describing lower bounds for the \(\m_k\)-function.

\begin{Definition}\label{taukdef}
Define the function \(\tau_k: \ZZ_{\geq 0} \to \ZZ_{\geq 0}\) 
by setting \(\tau_k(x)\) to be the unique non-negative integer such that
\begin{align*}
T_k(\tau_k(x) - 1) < x \leq T_k(\tau_k(x)).
\end{align*}
\end{Definition}

The following two lemmas are clear from definitions.

\begin{Lemma} \label{kappaTri}
For any $x \in \mathbb{Z}_{\geq 0}$,   we have $\tau_k(T_k(x))=x$.
\end{Lemma}
\begin{Lemma}\label{increasing}
For any $x\leq y$, we have
$\tau_k(x) \leq \tau_k(y).$
\end{Lemma}

We now prove some additional useful technical lemmas on \(T_k\) and \(\tau_k\). 

\begin{Lemma}\label{indTk}
For all \(x>0\), we have \(T_k(x)  = T_k(x-1) + T_{k-1}(x-1) + 1\).
\end{Lemma}
\begin{proof}
We have
\begin{align*}
1+ T_k(x-1) + T_{k-1}(x-1)  &= 
{x-1 \choose 0} + \sum_{i=1}^k { x-1 \choose i} + \sum_{i=1}^{k-1} {x-1 \choose i}\\
&=
\sum_{i=1}^k \left[ {x-1 \choose i} + {x -1 \choose i-1}\right]
= \sum_{i=1}^k {x \choose i}
= T_k(x),
\end{align*}
where the third equality follows from the binomial recurrence relation. 
\end{proof}

\begin{Lemma}\label{WrongWayLemmaBig} Let \(x,y \in \mathbb{Z}_{\geq 0}\) be such that \(y < T_{k-1}(\tau_k(x) - 2) +2\). Then \(\tau_k(x)-1 \leq \tau_k(x-y)\).
\end{Lemma}
\begin{proof}
We have by Lemma~\ref{indTk} that
\begin{align*}
T_k(\tau_k(x)-2) &= T_k(\tau_k(x)-1) - T_{k-1}(\tau_k(x)-2) - 1
<T_k(\tau_k(x)-1) - y+1.
\end{align*}
By the definition of \(\tau_k(x)\) we have \(T_k(\tau_k(x) - 1) < x\), so
\begin{align*}
T_k(\tau_k(x) - 2) \leq T_k(\tau_k(x) -1) - y < x-y.
\end{align*}
so by the definition of \(\tau_k(x-y)\), we have \(\tau_k(x-y)>\tau_k(x)-2\). Thus \(\tau_k(x-y) \geq \tau_k(x)-1\).
\end{proof}

\begin{Lemma}\label{DiffLem}
If \(x, y \in \mathbb{Z}_{\geq 0}\) and \(n \in \NN\) are such that \(x \equiv 0 \pmod n\) and \(T_k(\tau_k(x)) \equiv y \pmod n\), where \(0 \leq y < n\), then \(T_k(\tau_k(x)) - x \geq y\). 
\end{Lemma}
\begin{proof}
By definition, \(T_k(\tau_k(x)) \geq x\). Then \(T_k(\tau_k(x)) - x \equiv y \pmod n\), and \(T_k(\tau_k(x)) - x  \geq 0\), so \(T_k(\tau_k(x)) - x = y + nt\) for some \(t \in \mathbb{Z}_{\geq 0}\), so the result follows.
\end{proof}

\subsection{Triangular numbers}\label{TriResults}
Now we prove some technical lemmas in the case \(k=2\), recalling that \(T_2(x)\) is the triangular number \(1 + \cdots + x\). The next lemma is just a special case of Lemma~\ref{WrongWayLemmaBig}.

\begin{Lemma}\label{WrongWayLemma}
If \(x,y \in \mathbb{Z}_{\geq 0}\), with \(y< \tau_2(x)\), then \(\tau_2(x)-1 \leq \tau_2(x-y)\).
\end{Lemma}

\begin{Lemma}\label{TauLemma}
Let $x, y \in \mathbb{Z}_{\geq 0}, \ell \in \mathbb{N}$, with \(0 \leq y \leq x - \ell \tau_2(x) + T_2(\ell -1)\). Then we have
\begin{align*}
    \tau_2(y) \leq \tau_2(x)-\ell.
\end{align*}
\end{Lemma}
\begin{proof}
By Definition~\ref{taukdef}, we have
\begin{align*}
    y &\leq x - \ell \tau_2(x) + T_2(\ell - 1)
     \leq T_2(\tau_2(x)) - \ell \tau_2(x) + T_2(\ell - 1)\\
    &= [1+ \cdots + \tau_2(x)] - \ell\tau_2(x) + [1+2+\cdots +(\ell -1)] \\
    &= [1 + 2 + \cdots + \tau_2(x)] - [(\tau_2(x) - (\ell-1)) + \cdots + (\tau_2(x) - 1) + \tau_2(x)]\\
    &=1+ 2 + \cdots + (\tau_2(x) - \ell)
    =T_2(\tau_2(x) - \ell).
\end{align*}
Then, applying $\tau_2$ to both sides of the inequality, we have by Lemmas~\ref{kappaTri} and~\ref{increasing} that
\begin{align*}
    \tau_2(y) & \leq \tau_2(T_2(\tau_2(x) - \ell))
      = \tau_2(x) - \ell, 
\end{align*}
as desired. 
\end{proof}

\begin{Lemma}\label{BigTLem}
Let \(y, r,n \in \mathbb{N}\), with \(y \equiv r \pmod n\). Then:
\begin{align*}
T_2(y) \equiv \begin{cases}
T_2(r) + \frac{n}{2} \pmod n& \textup{if } n \equiv 0 \pmod 2, \frac{y-r}{n} \equiv 1 \pmod 2;\\
T_2(r) \pmod n&\textup{otherwise}.
\end{cases}
\end{align*}
\end{Lemma}
\begin{proof}
We may assume without loss of generality that \(y \geq r\). Note that since $y \equiv r \pmod{n}$, we have $y-r =n\ell$ for some $\ell \in \Z_{\geq 0}$. We prove the claim by induction on $\ell$. Let $\ell=0$. Then $y=r$ and $\frac{y-r}{n} \equiv 0 \pmod{2}$.
Therefore, \(  T_2({y}) = T_2({r}) \equiv T_2({r}) \pmod{n} \) so the base case holds.

Now assume \(\ell >0\) and the claim holds for all $\ell' < \ell$. Then
\begin{align*}
    T_2({y}) &= T_2({r +n\ell}) \\
    &= (r+n\ell)+(r+n\ell-1) + \cdots + (r + n\ell - (n-1)) + T_2({r+n(\ell-1)})\\
    &= nr - (0 + \cdots + (n-1)) + T_2({r + n(\ell-1)})\\
    &= nr - T_2({n-1}) + T_2({r + n(\ell-1)})\\
    &= nr -\frac{(n-1)n}{2} + T_2({r + n(\ell-1)}) \\
    & \equiv -\frac{(n-1)n}{2} + T_2({r + n(\ell-1)}) \pmod n
\end{align*}
We consider three separate cases, based on the parity of \(n\) and \(\ell\).

{\em Case 1.} Suppose $n$ is odd. Then we have that 
\(
    T_2({r+n(\ell-1)}) \equiv T_2({r}) \pmod{n}
\)
by the induction assumption. Therefore, 
\begin{align*}
    T_2({y}) 
    \equiv
    -\frac{(n-1)n}{2} + T_2({r + n(\ell-1)})
 \equiv -n  \cdot \frac{(n-1)}{2} + T_2({r})
    \equiv T_2({r}) \pmod{n}.
\end{align*}

{\em Case 2.} Suppose $n$ is even and $\ell$ is odd. Then $\ell-1$ is even, so we have $T_2({r+n(\ell-1)}) \equiv T_2(r) \pmod{n}$ by the induction assumption. Then 
\begin{align*}
    T_2(y) 
   & \equiv
    -\frac{(n-1)n}{2} + T_2({r + n(\ell-1)})
    \equiv -\frac{n}{2}(n-1)+ T_2(r) \pmod{n}\\
    &\equiv -\frac{n}{2}(-1) + T_2(r) 
    \equiv  T_2(r) + \frac{n}{2} \pmod{n}.
\end{align*}

{\em Case 3.} Suppose $n $ is even and $\ell$ is even. Then $\ell-1$ is odd, so we have $T_2({r+n(\ell-1)}) \equiv T_2(r) + \frac{n}{2} \pmod{n}$ by the induction assumption. Then
\begin{align*}
    T_2({y})  &\equiv -\frac{(n-1)n}{2} + T_2({r + n(\ell-1)}) \pmod{n}\\
   & \equiv-\frac{n}{2}(n-1) + T_2({r}) + \frac{n}{2} 
     \equiv n + T_2({r})  
    \equiv T_2({r}) \pmod{n}.
\end{align*}
Thus in any case, the claim holds for \(\ell\), completing the induction step and the proof.
\end{proof}

\section{Bounds on the \textup{\(\m_k\)}-function}
Now we establish bounds on the \(\m_k\)-function. 
The following lemma is clear from Definition~\ref{qkdef}.
\begin{Lemma}\label{isomq}
If \(\lambda \cong \nu\), then \(\m_k(\lambda) = \m_k(\nu)\).
\end{Lemma}

\begin{Theorem}\label{BOUNDS}
For all \(\lambda, k\), we have
\(
 \tau_k(|\lambda|) \leq \m_k(\lambda) \leq |\lambda|
\).
\end{Theorem}
\begin{proof}
We first prove that \(\m_k(\lambda) \leq |\lambda|\). The claim holds for \(k=1\) and \(\lambda = \varnothing\) by Definition~\ref{qkdef}. Now let \(k>1\), \(|\lambda|> 0\), and assume \(\m_{k'}(\lambda') \leq |\lambda'|\)  for all \(k' < k\), \(|\lambda'| < |\lambda|\). Let \(v\) be any maximal element in \(\lambda\). Then we have \(\lambda_{\succ v} = \varnothing\) and \(|\lambda_{\not \succeq v}| = |\lambda| - 1\), so:
\begin{align*}
\m_k(\lambda) &= \min \{ \max\{ \m_k(\lambda_{\not \succeq u}) , \m_{k-1}(\lambda_{\succ u})\}\mid u \in \lambda\}+1\\
&\leq  \max\{ \m_k(\lambda_{\not \succeq v}) , \m_{k-1}(\lambda_{\succ v})\} + 1
\leq \max\{ |\lambda_{\not \succeq v}|, 0\} + 1
 = (|\lambda| - 1) + 1
= |\lambda|,
\end{align*}
as desired.

Now we prove that \(\tau_k(|\lambda|) \leq \m_k(\lambda)\).
The claim holds for \(k =1\), as 
\(
\m_1(\lambda) = |\lambda| = {|\lambda| \choose 1} = T_1(|\lambda|),
\)
and the claim holds for \(\lambda = \varnothing\), as we have 
\(
\m_k(\varnothing) = 0 = \sum_{i=1}^k {0 \choose i} = T_k(0).
\)
Now let \(k>1\), \(|\lambda|> 0\), and assume \(\tau_{k'}(|\lambda'|) \leq \m_{k'}(\lambda')\)  for all \(k' < k\), \(|\lambda'| < |\lambda|\).
For some \(u \in \lambda\), we have
\begin{align*}
\m_k(\lambda)  = \max\{ \m_k(\lambda_{\not \succeq u}), \m_{k-1}(\lambda_{\succ u})\} + 1.
\end{align*}
Then by the induction assumption we have 
\begin{align}\label{firstineq}
\m_k(|\lambda|) \geq \m_k(|\lambda_{\not \succeq u}|) + 1 \geq \tau_k(|\lambda_{\not \succeq u}|) + 1
\end{align}
and
\begin{align}\label{secondineq}
\m_k(|\lambda|) \geq \m_{k-1}(|\lambda_{\succ u}|) + 1 \geq \tau_{k-1}(|\lambda_{\succ u}|) + 1.
\end{align}
Assume by way of contradiction that \(\m_k(|\lambda|) < \tau_k(|\lambda|)\). First we claim that \(|\lambda_{\succ u}| < T_{k-1}(\tau_k(|\lambda|) -2) + 1\). Indeed, if \(|\lambda_{\succ u}| \geq T_{k-1}(\tau_k(|\lambda|) -2) + 1\), then by Definition~\ref{taukdef} we would have \(\tau_{k-1}(|\lambda_{\succ u}|) > \tau_k(|\lambda|) - 2\), so \(\tau_{k-1}(|\lambda_{\succ u}|) \geq \tau_k(|\lambda|) -1\). But then
\begin{align*}
\tau_{k-1}(|\lambda_{\succ u}|) + 1 \geq \tau_{k}(|\lambda|) > \m_k(|\lambda|),
\end{align*}
a contradiction of (\ref{secondineq}). Thus \(|\lambda_{\succ u}| < T_{k-1}(\tau_k(|\lambda|) -2) + 1\) as desired.
Note then that \(|\lambda_{\succ u}| + 1 < T_{k-1}(\tau_k(|\lambda|) -2) + 2\), so by Lemma~\ref{WrongWayLemmaBig}, we have
\begin{align*}
\tau_k(|\lambda|) - 1 \leq \tau_k( |\lambda| - |\lambda_{\succ u}| - 1).
\end{align*}
Therefore, applying (\ref{firstineq}) we have
\begin{align*}
\m_k(|\lambda|) \geq \tau_k(|\lambda_{\not \succeq u}|) + 1 = \tau_k(|\lambda| - |\lambda_{\succ u}| - 1) + 1 \geq (\tau_k(|\lambda|) -1 ) + 1 = \tau_k(|\lambda|) > \m_k(|\lambda|),
\end{align*}
a contradiction. Therefore \(\tau_k(|\lambda|) \leq \m_k(|\lambda|)\), as desired. This completes the induction step, and the proof.
\end{proof}

With the following two lemmas, we prove that the bounds of Theorem~\ref{BOUNDS} are tight with respect to arbitrary posets.

\begin{Lemma}\label{TRIVIAL}
Let \(\lambda\) be a poset with trivial partial order. Then \(\m_k(\lambda) = |\lambda|\).
\end{Lemma}
\begin{proof}
If \(\lambda = \varnothing\) or \(k = 1\), the claim follows by Definition~\ref{qkdef}. Now let \(k>1\), \(|\lambda|> 0\), and assume \(\m_{k'}(\lambda') = |\lambda'|\)  for all \(k' < k\), and trivial posets \(\lambda'\) with \(|\lambda'| < |\lambda|\). Let \(u \in \lambda\). Then we have that \(\lambda_{\succ u} = \varnothing\), and \(\lambda_{\not \succeq u} = \lambda \backslash \{u\}\) is itself a trivial poset. Therefore by the induction assumption we have
\begin{align*}
\m_k(\lambda)& = \min \{ \max\{ \m_k(\lambda_{\not \succeq u}) , \m_{k-1}(\lambda_{\succ u})\}\mid u \in \lambda\}+1\\
&= \min \{ \max\{ |\lambda| -1, 0\} \mid u \in \lambda\} + 1 
= (|\lambda| -1) + 1 = |\lambda|,
\end{align*}
as desired.
\end{proof}

\begin{Lemma}\label{TOTAL}
Let \(\lambda\) be a totally ordered set. Then \(\m_k(\lambda) = \tau_k(|\lambda|)\).
\end{Lemma}
\begin{proof}
As usual, we note that the claim holds for \(k =1, n=0\) by Definition~\ref{qkdef}. 
We now let \(k>1\) and \(|\lambda|>0\), and make the induction assumption that \(\m_{k'}(\lambda') = \tau_{k'}(|\lambda'|)\) for all \(k' < k\) and totally ordered \(\lambda'\) with \(|\lambda'| < |\lambda|\). 

We may assume \(\lambda = [1,n]\), as any totally ordered set of cardinality \(n\) is equivalent to this interval. Note that we have \(0 \leq T_k(\tau_k(n) - 1) < n\) by Definition~\ref{taukdef}, so \(v := T_k(\tau_k(n) - 1) +1 \in [1,n]\). Then, applying Lemma~\ref{kappaTri}, we have
\begin{align*}
\m_k(\lambda_{\not \succeq v}) = \tau_k(|[1,v-1]|) = \tau_k(v-1) = \tau_k(T_k(\tau_k(n)-1)) = \tau_k(n) -1.
\end{align*}
On the other hand, we have
\begin{align*}
\m_{k-1}(\lambda_{\succ v}) = \tau_{k-1}(|[v+1,n]|) = \tau_{k-1}(n-v) = \tau_{k-1}(n - T_k(\tau_k(n)-1) - 1)).
\end{align*}
Then we have
\begin{align*}
\m_{k-1}(\lambda_{\succ v}) & =  \tau_{k-1}(n - T_k(\tau_k(n)-1) - 1)) 
\leq \tau_{k-1} (T_k(\tau_k(n)) - T_k(\tau_k(n)-1) -1))\\
&= \tau_{k-1} ((T_{k-1}(\tau_k(n)-1) + 1) -1)
=\tau_{k-1}(T_{k-1}(\tau_k(n) - 1))
= \tau_k(n)-1,
\end{align*}
using Lemma~\ref{increasing} and the fact that \(n \leq T_k(\tau_k(n))\) by Definition~\ref{taukdef} for the first inequality, Lemma~\ref{indTk} for the second equality, and Lemma~\ref{kappaTri} for the last equality.

Thus we have
\begin{align*}
\m_k(\lambda) &= \min \{ \max\{ \m_k(\lambda_{\not \succeq u}) , \m_{k-1}(\lambda_{\succ u})\} +1 \mid u \in \lambda\}\\
&\leq  \max\{ \m_k(\lambda_{\not \succeq v}) , \m_{k-1}(\lambda_{\succ v})\} +1
=(\tau_k(n) - 1) + 1
= \tau_k(n) = \tau_k(|\lambda|).
\end{align*}
Since \(\m_k(\lambda) \geq \tau_k(|\lambda|)\) by Theorem~\ref{BOUNDS}, we have \(\m_k(\lambda) = \tau_k(|\lambda|)\). This completes the induction step, and the proof.
\end{proof}

\begin{Remark}
The proof of Lemma~\ref{TOTAL} contains a solution to the strategy question from Problem~\ref{P1} for totally ordered sets, defined recursively for any \(k\in \NN\). Namely, one should query the element \(v\) such that \(|\lambda_{\prec v}| = T_k(\tau_k(|\lambda|) -1) \). If the query is negative, repeat the process with the totally ordered set \(\lambda_{\prec v}\). If the query is positive and \(k=1\), stop. Otherwise, repeat the process with the totally ordered set \(\lambda_{\succ v}\) and \(k: = k-1\). The final positive query will identify the element which generates the ideal \(\mu\).
\end{Remark}

\begin{Remark}
In view of Theorem~\ref{BOUNDS} and Lemmas~\ref{TRIVIAL} and~\ref{TOTAL}, one may be led to conjecture that \(\m_k(\lambda') \leq \m_k(\lambda)\) when \(\lambda'\) is a refinement of the poset \(\lambda\). This does not hold in general, however. For a counterexample, see Example~\ref{nonmon}, where the posets \(\lambda_0, \lambda_1, \lambda_2, \lambda_3\) are sequential refinements, but the corresponding sequence of \(\m_k\) values is not monotonic when \(k \geq 2\).
\end{Remark}

\section{Strategy in the \(k=2\) case}

We will now narrow our focus to the \(k=2\) setting. We develop a combinatorial language for describing query strategies in response to Problem~\ref{P1}. We fix some nonempty poset \(\lambda\) throughout this section.

\begin{Definition}
Let \(r \in \NN\), and \(\bu = (u_1, \ldots, u_r)\) be a sequence of elements of \(\lambda\). For each \(t = 1, \ldots, r\), define the subset:
\begin{align*}
\lambda_{\bu}^{(t)} := \lambda_{\succeq u_t} \backslash \lambda_{ \succeq \{u_1, \ldots, u_{t-1}\}} = \{v \in \lambda \mid v \succeq u_t, v \not \succeq u_i \textup{ for all } i = 1, \ldots, t-1\}.
\end{align*}
If \(\lambda_{\succeq \bu} = \lambda\) and \(\lambda_{\bu}^{(t)} \neq \varnothing\) for all \(t = 1, \ldots, r\), we call \(\bu\) a {\em \(\lambda\)-strategy}. 
\end{Definition}

By definition, the sets \(\lambda_{\bu}^{(1)}, \ldots, \lambda_{\bu}^{(r)}\) are mutually disjoint, so if \(\bu\) is a \(\lambda\)-strategy, we have:
\begin{align}\label{sqcupeq}
\lambda = \lambda_{\bu}^{(1)} \sqcup \cdots \sqcup \lambda_{\bu}^{(r)}.
\end{align}

\subsection{The \(\M_2\)-function}

\begin{Definition}\label{Q2def}
For a sequence of elements \(\bu = (u_1, \ldots, u_r)\) in \(\lambda\), we define:
\begin{align*}
\M_2(\lambda, \bu) := \max \{| \lambda_{\bu}^{(t)}| + t -1 \mid t= 1, \ldots, r \}.
\end{align*}
\end{Definition}
We will primarily be concerned with the value of \(\M_2(\lambda, \bu)\) when \(\bu\) is a \(\lambda\)-strategy. 

\begin{Example}\label{nonopex}
Let \(\lambda = \llbracket 5,7 \rrbracket\), and define the \(\lambda\)-strategy
\begin{align*}
\bu = ((2,6), (5,2), (1,5), (3,3), (2,1), (1,4), (1,1)).
\end{align*}
Then we may visually represent \(\bu\) in the diagram below:
\begin{align*}
\begin{array}{ccccc}
\hackcenter{
\begin{tikzpicture}[scale=0.55]
%
\fill[fill=white] (0,0)--(3,0)--(3,-0.2)--(0,-0.2)--(0,0);
\draw[thick, fill=lightgray!50]  (0,0)--(7,0)--(7,5)--(0,5)--(0,0);
\draw[thick, fill=blue!40!green!45]  (0,0)--(7,0)--(7,5)--(0,5)--(0,0);
\draw[  thick, fill=orange!60]  (3,0)--(7,0)--(7,5)--(3,5)--(3,0);
\draw[thick, thick,  fill=violet!35] (0,1)--(7,1)--(7,5)--(0,5)--(0,1);
\draw[thick,  fill=cyan!25] (2,2)--(7,2)--(7,5)--(2,5)--(2,2);
\draw[thick, fill=lime] (4,0)--(7,0)--(7,5)--(4,5)--(4,0);
\draw[thick, fill=yellow] (1,4)--(7,4)--(7,5)--(1,5)--(1,4);
\draw[thick, fill=pink] (5,1)--(7,1)--(7,5)--(5,5)--(5,1);
%
\draw[thick, dotted] (0,0)--(0,5);
\draw[thick, dotted] (1,0)--(1,5);
\draw[thick, dotted] (2,0)--(2,5);
\draw[thick, dotted] (3,0)--(3,5);
\draw[thick, dotted] (4,0)--(4,5);
\draw[thick, dotted] (5,0)--(5,5);
\draw[thick, dotted] (6,0)--(6,5);
\draw[thick, dotted] (7,0)--(7,5);
\draw[thick, dotted] (0,1)--(7,1);
\draw[thick, dotted] (0,2)--(7,2);
\draw[thick, dotted] (0,3)--(7,3);
\draw[thick, dotted] (0,4)--(7,4);
\draw[thick, dotted] (0,5)--(7,5);
\draw[thick]  (0,0)--(7,0)--(7,5)--(0,5)--(0,0);
\node at (5.5,1.5) { \textcircled{$\scriptstyle{1}$}};
\node at (1.5,4.5) { \textcircled{$\scriptstyle{2}$}};
\node at (4.5,0.5) { \textcircled{$\scriptstyle{3}$}};
\node at (2.5,2.5) { \textcircled{$\scriptstyle{4}$}};
\node at (0.5,1.5) { \textcircled{$\scriptstyle{5}$}};
\node at (3.5,0.5) { \textcircled{$\scriptstyle{6}$}};
\node at (0.5,0.5) { \textcircled{$\scriptstyle{7}$}};
\end{tikzpicture}
}
\\
\scriptstyle{\textup{{The poset \(\lambda =\llbracket 5,7 \rrbracket\) with \(\lambda\)-strategy \(\bu\) }}} 
\end{array}
\end{align*}
The elements \(u_1, \ldots, u_7\) are marked with circled numbers. For each \(i \in \{1, \ldots, 7\}\), \(\lambda_{\bu}^{(i)}\) is the set of boxes in the same colored region as the box marked \({ \textcircled{$\scriptstyle{i}$}}\). The cardinalities of these sets are \(8,4,6,4,9,1,3\) respectively, so we have
\begin{align*}
\M_2(\lambda, \bu) = \max \{ 8+0, 4+1, 6+2, 4+3, 9+4, 1+5, 3+6 \} = 13.
\end{align*}
\end{Example}

We consider now some special choices of \(\lambda\)-strategies.

\begin{Lemma}\label{EveryShot}
For any nonempty poset \(\lambda\), let \(\bu = (u_1, \ldots, u_{|\lambda|})\) be any arrangement of the elements of \(\lambda\) which is non-increasing with respect to the partial order. Then \(\bu\) is a \(\lambda\)-strategy and \textup{\(\M_2(\lambda, \bu) = |\lambda|\)}.
\end{Lemma}
\begin{proof}
By the condition on \(\bu\) we have \(|\lambda_{\bu}^{(t)}| = 1\) for all \(t\), so \(\bu\) is a \(\lambda\)-strategy and
\begin{align*}
\M_2(\lambda, \bu) &= \max \{  | \lambda_{\bu}^{(t)}| + t-1\mid t= 1, \ldots, |\lambda|\}
= \max \{ 1 + t -1 \mid t= 1, \ldots, |\lambda|\}= |\lambda|,
\end{align*}
as desired.
\end{proof}

\begin{Lemma}\label{SingleShot}
Let \(\lambda\) be a nonempty poset, and assume there exists a \(\lambda\)-strategy \(\bu = (u_1)\) of length one. Then we have \(\textup{\(\M_2(\lambda, \bu) = |\lambda|\)}\).
\end{Lemma}
\begin{proof}
By the definition of \(\lambda\)-strategies \(\bu\), we must have \(\lambda = \lambda_{\succeq \bu} = \lambda_{\succeq u_1} =  \lambda_{\bu}^{(1)} \). Thus we have
\(
\M_2(\lambda, \bu) = |  \lambda_{\bu}^{(1)}| =  |\lambda|,
\)
as desired.
\end{proof}

For sequences of elements \(\bv = (v_1, \ldots, v_s)\) and \(\bw = (w_1, \ldots, w_r)\) in \(\lambda\), we will write \(\bv\bw\) for the concatenation \((v_1, \ldots, v_s, w_1, \ldots, w_r)\), or just \(v_1\bw\) if \(\bv = (v_1)\). 
For \(u \in \lambda\) with \(\lambda_{\succeq u} \neq \lambda\), note that \(u \bw\) is a \(\lambda\)-strategy if and only if \(\bw\) is a \(\lambda_{\not \succeq u}\)-strategy.

\begin{Lemma}\label{truncatelem}
Let \(\lambda\) be a nonempty poset. Let \(\bv = (v_1, \ldots, v_s)\) be a sequence of elements of \(\lambda\), and \(\bw = (w_1, \ldots, w_r)\) be a sequence of elements of \(\lambda_{\not \succeq \bv}\). Then, setting \(\bu = \bv\bw\), we have
\begin{align*}
\textup{\(\M_2(\lambda, \bu) = \max\{ \M_2(\lambda, \bv), \M_2(\lambda_{\not \succeq \bv}, \bw) + s\}.\)}
\end{align*}
\end{Lemma}
\begin{proof}
Note that for \(t = 1, \ldots, s\), we have \(\lambda_{\bu}^{(t)} = \lambda_{\bv}^{(t)}\), and for \(t= s+1, \ldots, s+r\), we have
\begin{align*}
\lambda_{\bu}^{(t)} &= \lambda_{\succeq u_t} \backslash \lambda_{ \succeq \{u_1, \ldots, u_{t-1}\}}
=
(\lambda_{\not \succeq \{u_1, \ldots, u_s\}})_{\succeq u_t} \backslash (\lambda_{\not \succeq \{u_1, \ldots, u_s\}})_{ \succeq \{u_{s+1}, \ldots, u_{t-1}\}}\\
&= (\lambda_{\not \succeq \{v_1, \ldots, v_s\}})_{\succeq w_{t-s}} \backslash (\lambda_{ \succeq \{v_1, \ldots, v_s\}})_{ \succeq\{w_1, \ldots, w_{t-s-1}\}}
=
(\lambda_{\not \succeq \bv})_{\bw}^{(t-s)}.
\end{align*}
Thus we have
\begin{align*}
\M_2(\lambda, \bu)
&=
\max\{ |\lambda_{\bu}^{(t)}| + t - 1 \mid t = 1, \ldots, s+r \}\\
&=
\max \{ \max\{|\lambda_{\bu}^{(t)}| + t -1\mid t=1, \ldots, s\}, \max \{|\lambda_{\bu}^{(t)}| + t - 1 \mid u=s+1, \ldots, s+r\}\}\\
&=\max\{ \max\{|\lambda_{\bv}^{(t)}| + t - 1 \mid t=1, \ldots, r\}, \max\{|(\lambda_{\not \succeq \bv})_{\bw}^{(t-s)}| + t - 1 \mid u=s+1, \ldots, s+r\}\}\\
&=\max\{ \M_2(\lambda, \bv),  \max\{|(\lambda_{\not \succeq \bv})_{\bw}^{(t)}| + t +s - 1 \mid t=1, \ldots, r\}\}\\
&=\max\{ \M_2(\lambda, \bv),  \max\{|(\lambda_{\not \succeq \bv})_{\bw}^{(t)}| + t  - 1 \mid t=1, \ldots, r\} +s\}\\
&= \max \{ \M_2(\lambda, \bv), \M_2(\lambda_{\not \succeq \bv}, \bw) +s\},
\end{align*}
as desired.
\end{proof}

\subsection{
Connecting \(\M_2\) and \(\m_2\)
}

\begin{Theorem}\label{sameq}
Let \(\lambda\) be a nonempty poset. We have
\begin{align}\label{newq2}
\textup{\(\m\)}_2(\lambda) = \min \{\textup{\(\M\)}_2(\lambda, \bu) \mid \textup{\(\bu\) a \(\lambda\)-strategy}\}.
\end{align}
\end{Theorem}
\begin{proof}
We go by induction on \(|\lambda|\). The base case \(|\lambda| = 1\) follows immediately from Lemma~\ref{SingleShot}. Now assume \(|\lambda| > 1\) and the claim holds for all \(|\lambda'| < |\lambda|\). Note that by Lemmas~\ref{EveryShot} and~\ref{SingleShot}, it suffices to take the minimum on the right of (\ref{newq2}) over \(\lambda\)-strategies of length greater than one. Thus we have
\begin{align*}
\min\{ \M_2(\lambda,& \bu) \mid \textup{\(\bu\) a \(\lambda\)-strategy}\}\\
&= \min\{ \M_2(\lambda, \bu) \mid \textup{\(\bu\) a \(\lambda\)-strategy of length greater than one}\}\\
&= \min\{ \M_2(\lambda, u\bw) \mid u \in \lambda, \textup{\(u\bw\) a \(\lambda\)-strategy}\}\\
&=\min\{ \M_2(\lambda, u\bw) \mid u \in \lambda, \textup{\(\bw\) a \(\lambda_{\not \succeq u}\)-strategy}\}\\
&= \min\{ \max\{ \M_2(\lambda, (u)), \M_2(\lambda_{\not \succeq u}, \bw)+1\} \mid u \in \lambda,  \textup{\(\bw\) a \(\lambda_{\not \succeq u}\)-strategy}\}\\
&= \min\{ \max\{ |\lambda_{\succeq u}|, \M_2(\lambda_{\not \succeq u}, \bw)+1\} \mid u \in \lambda, \textup{\(\bw\) a \(\lambda_{\not \succeq u}\)-strategy}\}\\
&= \min\{ \max\{ |\lambda_{\succ u}| + 1, \M_2(\lambda_{\not \succeq u}, \bw)+1\} \mid u \in \lambda,  \textup{\(\bw\) a \(\lambda_{\not \succeq u}\)-strategy}\}\\
&= \min\{ \max\{\m_1(\lambda_{\succ u}) + 1, \M_2(\lambda_{\not \succeq u}, \bw)+1\} \mid u \in \lambda, \textup{\(\bw\) a \(\lambda_{\not \succeq u}\)-strategy}\}\\
&=\min\{
\min\{
 \max\{\m_1(\lambda_{\succ u}) + 1, \M_2(\lambda_{\not \succeq u}, \bw)+1\}
 \mid
 \textup{\(\bw\) a \(\lambda_{\not \succeq u}\)-strategy}
 \}
 \mid
 u \in \lambda\}\\
&= \min\{ \max\{\m_1(\lambda_{\succ u}) + 1, 
\min\{\M_2(\lambda_{\not \succeq u}, \bw)
\mid
\textup{\(\bw\) a \(\lambda_{\not \succeq u}\)-strategy}
\} +1\}\mid u \in \lambda\}\\
&=
\min\{
\max\{ 
\m_1(\lambda_{\succ u}) + 1, \m_2(\lambda_{\not \succeq u}) + 1 \} \mid u \in \lambda\}\\
&=
\min\{ \max\{  \m_2(\lambda_{\not \succeq u}), \m_1(\lambda_{\succ u})\} \mid u \in \lambda\} + 1\\
&= \m_2(\lambda).
\end{align*}
The fourth equality above follows from Lemma~\ref{truncatelem}, and the tenth equality follows from the induction assumption. This completes the induction step, and the proof.
\end{proof}

\subsection{Some examples}
Combining Theorems~\ref{BOUNDS} and~\ref{sameq} can be a useful method of computing \(\m_2(\lambda)\), as shown in the examples below.

\begin{Example}\label{SolvedEx1}
Let \(\lambda = \llbracket 5,7\rrbracket\), and consider the \(\lambda\)-strategy:
\begin{align*}
\bu = ((4,4), (2,5), (1,4), (1,3), (4,1), (1,2), (2,1), (1,1)).
\end{align*}
Then, as in Example~\ref{nonopex}, we visually represent \(\bu\) in the diagram below:
\begin{align*}
\begin{array}{ccccc}
\hackcenter{
\begin{tikzpicture}[scale=0.55]
%
\fill[fill=white] (0,0)--(3,0)--(3,-0.2)--(0,-0.2)--(0,0);
\draw[thick, fill=lightgray!50]  (0,0)--(7,0)--(7,5)--(0,5)--(0,0);
\draw[thick, fill=brown!35]  (0,0)--(7,0)--(7,5)--(0,5)--(0,0);
\draw[thick, fill=blue!40!green!45]  (0,1)--(7,1)--(7,5)--(0,5)--(0,1);
\draw[thick, fill=orange!60]  (1,0)--(7,0)--(7,5)--(1,5)--(1,0);
\draw[thick,  fill=violet!35]  (0,3)--(7,3)--(7,5)--(0,5)--(0,3);
\draw[thick, fill=cyan!25]  (2,0)--(7,0)--(7,5)--(2,5)--(2,0);
\draw[thick, fill=lime]  (3,0)--(7,0)--(7,5)--(3,5)--(3,0);
\draw[thick, fill=yellow]  (4,1)--(7,1)--(7,5)--(4,5)--(4,1);
\draw[thick, fill=pink]  (3,3)--(7,3)--(7,5)--(3,5)--(3,3);
%
\draw[thick, dotted] (0,0)--(0,5);
\draw[thick, dotted] (1,0)--(1,5);
\draw[thick, dotted] (2,0)--(2,5);
\draw[thick, dotted] (3,0)--(3,5);
\draw[thick, dotted] (4,0)--(4,5);
\draw[thick, dotted] (5,0)--(5,5);
\draw[thick, dotted] (6,0)--(6,5);
\draw[thick, dotted] (7,0)--(7,5);
\draw[thick, dotted] (0,1)--(7,1);
\draw[thick, dotted] (0,2)--(7,2);
\draw[thick, dotted] (0,3)--(7,3);
\draw[thick, dotted] (0,4)--(7,4);
\draw[thick, dotted] (0,5)--(7,5);
\draw[thick]  (0,0)--(7,0)--(7,5)--(0,5)--(0,0);
\node at (3.5,3.5){ \textcircled{$\scriptstyle{1}$}};
\node at (4.5,1.5){ \textcircled{$\scriptstyle{2}$}};
\node at (3.5,0.5){ \textcircled{$\scriptstyle{3}$}};
\node at (2.5,0.5){ \textcircled{$\scriptstyle{4}$}};
\node at (0.5,3.5){ \textcircled{$\scriptstyle{5}$}};
\node at (1.5,0.5){ \textcircled{$\scriptstyle{6}$}};
\node at (0.5,1.5){ \textcircled{$\scriptstyle{7}$}};
\node at (0.5,0.5){ \textcircled{$\scriptstyle{8}$}};
\end{tikzpicture}
}
\\
\scriptstyle{\textup{{The poset \(\lambda = \llbracket 5,7 \rrbracket\) with \(\lambda\)-strategy \(\bu\)}}} 
\end{array}
\end{align*}
This gives 
\begin{align*}
\M_2(\lambda, \bu) = \max\{ 8+0, 6+1, 6+2, 5+3, 4+4, 3+5, 2+6, 1+7\} = 8.
\end{align*}
Thus by Theorem~\ref{sameq} we have \(\m_2(\lambda) \leq 8\). But by Theorem~\ref{BOUNDS} we also have
\begin{align*}
\m_2(\lambda) \geq \tau_2(|\lambda|) = \tau_2(35) = 8,
\end{align*}
so \(\m_2(\lambda) = 8\).
\end{Example}

\begin{Example}\label{66ex}
Let \(\lambda = \llbracket 6,6 \rrbracket\). As \(|\lambda| = 36\), any \(\lambda\)-strategy \(\bu = (u_1, \ldots, u_r)\) which satisfies \(\M_2(\lambda, \bu) = \tau_2(|\lambda|) = 8\) must have \(r=8\) and \(|\lambda_{\bu}^{(t)}| = 9-t\) for all \(t = 1, \ldots, 8\). It is straightforward to check that no such \(\lambda\)-strategy exists, so by Theorems~\ref{BOUNDS} and~\ref{sameq}, we have \(\m_2(\lambda, \bu) > 8\). Now consider the \(\lambda\)-strategy :
\begin{align*}
\bv = ((5,3),(4,2),(2,4),(3,1),(1,4),(2,1),(1,2),(1,1)).
\end{align*}
We visually represent \(\bv\) in the diagram:
\begin{align*}
\begin{array}{c}
\hackcenter{
\begin{tikzpicture}[scale=0.58]
%
\fill[fill=white] (0,0)--(3,0)--(3,-0.2)--(0,-0.2)--(0,0);
\draw[thick, fill=brown!35]  (0,0)--(6,0)--(6,6)--(0,6)--(0,0);
\draw[ thick,  fill=blue!40!green!45]  (1,0)--(6,0)--(6,6)--(1,6)--(1,0);
\draw[ thick,  fill=orange!60]  (0,1)--(6,1)--(6,6)--(0,6)--(0,1);
\draw[thick,  fill=violet!35]  (3,0)--(6,0)--(6,6)--(3,6)--(3,0);
\draw[thick,  fill=cyan!25]  (0,2)--(6,2)--(6,6)--(0,6)--(0,2);
\draw[thick, fill=lime]  (3,1)--(6,1)--(6,6)--(3,6)--(3,1);
\draw[thick, fill=yellow]  (1,3)--(6,3)--(6,6)--(1,6)--(1,3);
\draw[thick, fill=pink]  (2,4)--(6,4)--(6,6)--(2,6)--(2,4);
\draw[thick, dotted] (0,0)--(0,6);
\draw[thick, dotted] (1,0)--(1,6);
\draw[thick, dotted] (2,0)--(2,6);
\draw[thick, dotted] (3,0)--(3,6);
\draw[thick, dotted] (4,0)--(4,6);
\draw[thick, dotted] (5,0)--(5,6);
\draw[thick, dotted] (0,1)--(6,1);
\draw[thick, dotted] (0,2)--(6,2);
\draw[thick, dotted] (0,3)--(6,3);
\draw[thick, dotted] (0,4)--(6,4);
\draw[thick, dotted] (0,5)--(6,5);
\node at (2.5,4.5){ \textcircled{$\scriptstyle{1}$}};
\node at (1.5,3.5){ \textcircled{$\scriptstyle{2}$}};
\node at (3.5,1.5){ \textcircled{$\scriptstyle{3}$}};
\node at (0.5,2.5){ \textcircled{$\scriptstyle{4}$}};
\node at (3.5,0.5){ \textcircled{$\scriptstyle{5}$}};
\node at (0.5,1.5){ \textcircled{$\scriptstyle{6}$}};
\node at (1.5,0.5){ \textcircled{$\scriptstyle{7}$}};
\node at (0.5,0.5){ \textcircled{$\scriptstyle{8}$}};
%
%
\end{tikzpicture}
}
\\
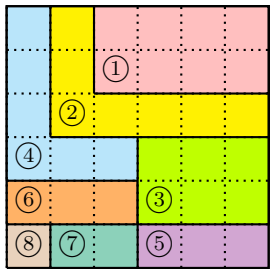
\scriptstyle{\textup{{The poset \(\lambda = \llbracket 6,6 \rrbracket\) with \(\lambda\)-strategy \(\bv\)}}} 
\end{array}
\end{align*}
This gives \(\M_2(\lambda, \bv) = 9\), so it follows from Theorem~\ref{sameq} that \(\m_2(\lambda) = 9\).
\end{Example}

\subsection{Strategies for Problem~\ref{P1} in the \(k=2\) case}
We now relate these definitions and results back to Problem~\ref{P1}, in the case where only two positive query results are permitted. Recall as in \S\ref{qkexp} that we have the unknown ideal \(\mu = \varnothing\) or \(\mu = \lambda_{\preceq x}\) for some \(x \in \lambda\). The \(\lambda\)-strategy \(\bu = (u_1, \ldots, u_r)\) defines a search strategy for \(\mu\) as follows.

We query the elements \(u_1, u_2, \ldots\) in sequence, until we have a positive query. If all the queries are negative, then, since \(\lambda_{\succeq \bu} = \lambda\), we have that \(\mu = \varnothing\), and we are done after \(r \leq |\lambda_{\bu}^{(r)}| + r -1\) queries. 
Assume the query of \(u_t\) is positive. Then the element \(x\) is known to belong to \(\lambda_{\succeq u_t}\), and known to not belong to \(\lambda_{\succeq \{u_1, \ldots, u_{t-1}\}}\). Thus \(x\) may be any of the elements in \(\lambda_{\bu}^{(t)}\). With one positive query remaining, the elements in \(\lambda_{\bu}^{(t)} \backslash \{u_t\}\) must be sequentially queried in any non-increasing order, as in \S\ref{k1case}. Thus, when the \(u_t\) query is positive, \(|\lambda_{\bu}^{(t)}| +t -1\) total queries are necessary to guarantee identification of \(\mu\).

Therefore, by Definition~\ref{Q2def}, the value \(\M_2(\lambda, \bu)\) represents the maximum number of queries necessary to identify \(\mu\) via the search strategy defined by \(\bu\). Thus, in view of Theorem~\ref{sameq}, we may reframe the \(k=2\) case of Problem~\ref{P1} in this combinatorial language: 

\begin{pppp}
Find the value \(\m_2(\lambda)\), and identify a \(\lambda\)-strategy \(\bu\) such that \(\M_2(\lambda, \bu) = \m_2(\lambda)\). 
\end{pppp}

\section{Product posets of finite totally ordered sets}

If \(u=(a,b) \in \NN^2\), we define the {\em transpose} element \(u^T := (b,a)\).
We extend this definition to sequences of elements \(\bu = (u_1,\ldots, u_r)\) in \(\NN^2\) and subsets \(S \subset \NN^2\) by setting:
\begin{align*}
\bu^T := (u_1^T, \ldots, u_r^T),
\qquad
\qquad
S^T := \{s^T \mid s \in S\}
\end{align*}
The transpose map induces an isomorphism of posets \(\llbracket m,n \rrbracket \cong \llbracket n,m \rrbracket\), for all \(m,n \in \NN\).

In this section it will be convenient to make use of a horizontally compressed visual shorthand for sequences of elements \(\bv = (v_1, \ldots, v_r)\) in \(\lambda = \llbracket m,n \rrbracket\). Using the `box array' representation of \(\llbracket m,n \rrbracket\), we will label the element \(v_i\) with \( \textcircled{$\scriptstyle{i}$}\) as usual, and then label every row in \(\lambda_{\bv}^{(i)}\) with the number of elements in that row. This visual information is sufficient to describe exactly all elements \(v_i\) in \(\bv\), and the related sets \(\lambda_{\bu}^{(i)}\).

\begin{Example}
Let \(\lambda = \llbracket 3,17 \rrbracket\). If \(\bv = ((3,9), (2,13), (2,6), (1,15), (3,2), (1,4))\), then below we have the explicit visual representation of \(\bv\) (on the left) and the compressed shorthand representation of \(\bv\) (on the right).
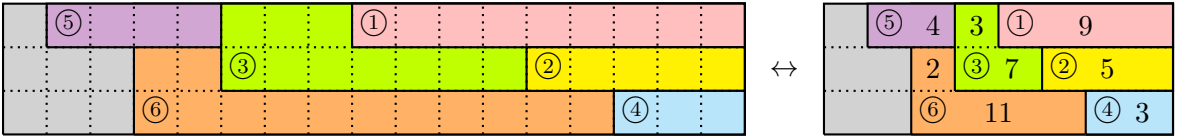
\begin{align*}
{}
\hackcenter{
\begin{tikzpicture}[scale=0.58]
\draw[ thick, fill=lightgray!70]  (0,0)--(17,0)--(17,3)--(0,3)--(0,0); 
\draw[thick,  fill=orange!60]  (3,0)--(17,0)--(17,3)--(3,3)--(3,0); 
\draw[thick,  fill=violet!35]  (1,2)--(17,2)--(17,3)--(1,3)--(1,2); 
\draw[ thick, fill=lime]  (5,1)--(17,1)--(17,3)--(5,3)--(5,1); 
\draw[ thick, fill=pink]  (8,2)--(17,2)--(17,3)--(8,3)--(8,2); 
\draw[ thick, fill=yellow]  (12,1)--(17,1)--(17,2)--(12,2)--(12,1); 
\draw[thick,  fill=cyan!25]  (14,0)--(17,0)--(17,1)--(14,1)--(14,0); 
\draw[thick, dotted] (0,1)--(17,1);
\draw[thick, dotted] (0,2)--(17,2);
\draw[thick, dotted] (1,0)--(1,3);
\draw[thick, dotted] (2,0)--(2,3);
\draw[thick, dotted] (3,0)--(3,3);
\draw[thick, dotted] (4,0)--(4,3);
\draw[thick, dotted] (5,0)--(5,3);
\draw[thick, dotted] (6,0)--(6,3);
\draw[thick, dotted] (7,0)--(7,3);
\draw[thick, dotted] (8,0)--(8,3);
\draw[thick, dotted] (9,0)--(9,3);
\draw[thick, dotted] (10,0)--(10,3);
\draw[thick, dotted] (11,0)--(11,3);
\draw[thick, dotted] (12,0)--(12,3);
\draw[thick, dotted] (13,0)--(13,3);
\draw[thick, dotted] (14,0)--(14,3);
\draw[thick, dotted] (15,0)--(15,3);
\draw[thick, dotted] (16,0)--(16,3);
\node at (8.5,2.5) { \textcircled{$\scriptstyle{1}$}};
\node at (12.5,1.5) { \textcircled{$\scriptstyle{2}$}};
\node at (14.5,0.5) { \textcircled{$\scriptstyle{4}$}};
\node at (5.5,1.5) { \textcircled{$\scriptstyle{3}$}};
\node at (1.5,2.5) { \textcircled{$\scriptstyle{5}$}};
\node at (3.5,0.5) { \textcircled{$\scriptstyle{6}$}};
\end{tikzpicture}
}
\;\;
\leftrightarrow
\;\;
\hackcenter{
\begin{tikzpicture}[scale=0.58]
\draw[ thick, fill=lightgray!70]  (4,0)--(12,0)--(12,3)--(4,3)--(4,0); 
\draw[thick,  fill=orange!60]  (6,0)--(12,0)--(12,3)--(6,3)--(6,0); 
\draw[ thick,  fill=violet!35]  (5,2)--(12,2)--(12,3)--(5,3)--(5,2); 
\draw[ thick, fill=lime]  (7,1)--(12,1)--(12,3)--(7,3)--(7,1); 
\draw[ thick, fill=pink]  (8,2)--(12,2)--(12,3)--(8,3)--(8,2); 
\draw[ thick, fill=yellow]  (9,1)--(12,1)--(12,2)--(9,2)--(9,1); 
\draw[thick,  fill=cyan!25]  (10,0)--(12,0)--(12,1)--(10,1)--(10,0); 
\draw[thick, dotted] (4,1)--(12,1);
\draw[thick, dotted] (4,2)--(12,2);
\node at (8.5,2.5) { \textcircled{$\scriptstyle{1}$}};
\node at (9.5,1.5) { \textcircled{$\scriptstyle{2}$}};
\node at (10.5,0.5) { \textcircled{$\scriptstyle{4}$}};
\node at (7.5,1.5) { \textcircled{$\scriptstyle{3}$}};
\node at (5.5,2.5) { \textcircled{$\scriptstyle{5}$}};
\node at (6.5,0.5) { \textcircled{$\scriptstyle{6}$}};
\node[] at (10,2.5) {$9$};
\node[] at (11.3,0.5) {$3$};
\node[] at (10.5,1.5) {$5$};
\node[] at (7.5,2.5) {$3$};
\node[] at (8,0.5) {$11$};
\node[] at (8.3,1.5) {$7$};
\node[] at (6.5,2.5) {$4$};
\node[] at (6.5,1.5) {$2$};
\end{tikzpicture}
}
\end{align*}
\end{Example}

Now we prove the second main theorem of this paper.

\begin{Theorem}\label{mainthm}
Let \(\lambda = \llbracket m,n\rrbracket\), with \(m \leq 6\) or \(n \leq 6\). Then we have:
\begin{align*}
\textup{\(\m_2(\lambda)\)} = \begin{cases}
9 & \textup{if }m=n=6;\\
\tau_2(mn) & \textup{otherwise}.
\end{cases}
\end{align*}
Moreover, Algorithm~\ref{Alg} below produces an explicit \(\lambda\)-strategy \(\bu\) such that \textup{\(\M_2(\lambda, \bu) =\m_2(\lambda)\)}. 
\end{Theorem}

\begin{Algorithm}\label{Alg} \(\) We assume \(\lambda = \llbracket m,n \rrbracket\), with one of \(m,n \) less than or equal to 6. This algorithm produces a \(\lambda\)-strategy \(\bu\) such that \(\M_2(\lambda, \bu) = \m_2(\lambda)\).\\

\noindent({\tt Step} 0) Let \(\bu = () \) be the empty sequence. Go to \(({\tt Step }\;1)\).\\

\noindent({\tt Step} 1)  If the number of columns of \(\lambda\) is greater than the number of rows, then redefine \(\lambda:=\lambda^{T}\), and set \({\tt flip}=1\). Otherwise set \({\tt flip}=0\). Redefine \(m,n\) if necessary such that \(\lambda = \llbracket m,n \rrbracket\). Go to \(({\tt Step } \;m+1)\).\\

\noindent({\tt Step} 2,  \(\lambda = \llbracket 1,n\rrbracket)\). Define \(t := \tau_2(n)\). Define \(\bv\) to be the one-element sequence in \(\lambda\) depicted below. Go to ({\tt Step} 8).
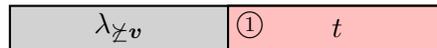
\begin{align*}
\begin{array}{c}
\hackcenter{
\begin{tikzpicture}[scale=0.58]
%
\draw[thick, fill=lightgray!70]  (9,0)--(19,0)--(19,1)--(9,1)--(9,0); 
\draw[ thick, fill=pink]  (14,0)--(19,0)--(19,1)--(14,1)--(14,0); 
\node at (14.5,0.5){ \textcircled{$\scriptstyle{1}$}};
\node at (11.5,0.5){ $\lambda_{\not \succeq \bv}$};
\node[] at (16.5,0.5) {$ t$};
\end{tikzpicture}
}
\end{array}
\end{align*}
\\

\noindent({\tt Step} 3,  \(\lambda = \llbracket 2,n \rrbracket)\). Define \(t := \tau_2(2n)\). Define \(\bv\) to be the element sequence  in \(\lambda\) depicted below which corresponds to the appropriate condition on \(t\). Go to ({\tt Step} 8).
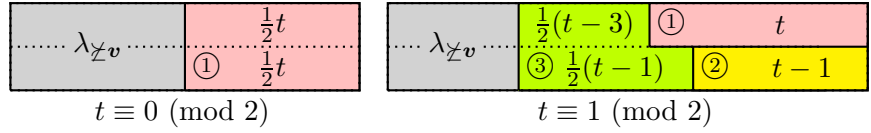
\begin{align*}
\begin{array}{cc}
\hackcenter{
\begin{tikzpicture}[scale=0.58]
%
\draw[thick, fill=lightgray!70]  (10,0)--(18,0)--(18,2)--(10,2)--(10,0); 
\draw[ thick, fill=pink]  (14,0)--(18,0)--(18,2)--(14,2)--(14,0); 
\draw[thick, dotted] (10,0)--(18,0)--(18,2)--(10,2)--(10,0);
\draw[thick, dotted] (10,1)--(11.3,1);
\draw[thick, dotted] (12.7,1)--(18,1);
\node at (14.5,0.5){ \textcircled{$\scriptstyle{1}$}};
\node at (12,1){ $\lambda_{\not \succeq \bv}$};
\node[] at (16,0.5) {$ \frac{1}{2}t $};
\node[] at (16,1.5) {$\frac{1}{2}t $};
\end{tikzpicture}
}
&
\hackcenter{
\begin{tikzpicture}[scale=0.58]
%
\draw[thick, fill=lightgray!70]  (9,0)--(19,0)--(19,2)--(9,2)--(9,0); 
\draw[ thick, fill=lime]  (12,0)--(20,0)--(20,2)--(12,2)--(12,0); 
\draw[ thick, fill=yellow]  (16,0)--(20,0)--(20,2)--(16,2)--(16,0); 
\draw[ thick, fill=pink]  (15,1)--(20,1)--(20,2)--(15,2)--(15,1); 
\draw[thick, dotted] (9,0)--(20,0)--(20,2)--(9,2)--(9,0);
\draw[thick, dotted] (9,1)--(9.8,1);
\draw[thick, dotted] (11.2,1)--(19,1);
\node at (15.5,1.5){ \textcircled{$\scriptstyle{1}$}};
\node at (16.5,0.5){ \textcircled{$\scriptstyle{2}$}};
\node at (12.5,0.5){ \textcircled{$\scriptstyle{3}$}};
\node at (10.5,1){ $\lambda_{\not \succeq \bv}$};
\node[] at (18,1.5) {$t $};
\node[] at (18.5,0.5) {$t-1$};
\node[] at (13.5,1.5) {$\frac{1}{2}(t-3)$};
\node[] at (14.2,0.5) {$\frac{1}{2}(t-1)$};
\end{tikzpicture}
}
\\
t \equiv 0 \pmod 2
&
t \equiv 1 \pmod 2
\end{array}
\end{align*}
\\

\noindent({\tt Step} 4,  \(\lambda = \llbracket 3,n\rrbracket)\). Define \(t := \tau_2(3n)\). Define \(\bv\) to be the element sequence  in \(\lambda\) depicted below which corresponds to the appropriate condition on \(t\). Go to ({\tt Step} 8). 
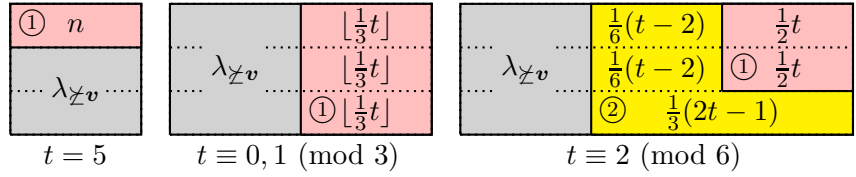
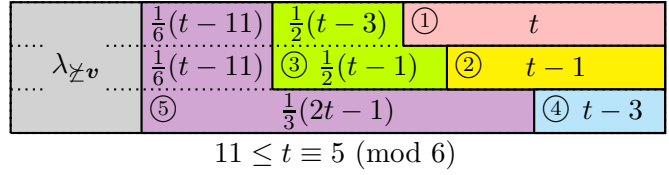
\begin{align*}
\begin{array}{ccc}
{}
\hackcenter{
\begin{tikzpicture}[scale=0.58]
\draw[ thick, fill=lightgray!70]  (2,1)--(5,1)--(5,4)--(2,4)--(2,1); 
\draw[ thick, fill=pink]  (2,3)--(5,3)--(5,4)--(2,4)--(2,3); 
\draw[thick, dotted]  (2,1)--(5,1)--(5,4)--(2,4)--(2,1); 
\draw[thick, dotted] (2,2)--(2.8,2);
\draw[thick, dotted] (4.2,2)--(5,2);
\draw[thick, dotted] (2,3)--(5,3);
\node at (2.5,3.5){ \textcircled{$\scriptstyle{1}$}};
\node at (3.5,2){ $\lambda_{\not \succeq \bv}$};
\node[] at (3.5,3.5) {$n$};
\end{tikzpicture}
}
&
{}
\hackcenter{
\begin{tikzpicture}[scale=0.58]
%
\draw[thick, fill=lightgray!70]  (11,0)--(17,0)--(17,3)--(11,3)--(11,0); 
\draw[ thick, fill=pink]  (14,0)--(17,0)--(17,3)--(14,3)--(14,0); 
\draw[thick, dotted] (11,0)--(17,0)--(17,3)--(11,3)--(11,0);
\draw[thick, dotted] (11,1)--(11.8,1);
\draw[thick, dotted] (13.2,1)--(17,1);
\draw[thick, dotted] (11,2)--(11.8,2);
\draw[thick, dotted] (13.2,2)--(17,2);
\node at (14.5,0.5){ \textcircled{$\scriptstyle{1}$}};
\node at (12.5,1.5){ $\lambda_{\not \succeq \bv}$};
\node[] at (15.5,0.5) {$\lfloor \frac{1}{3}t\rfloor $};
\node[] at (15.5,1.5) {$\lfloor \frac{1}{3}t\rfloor $};
\node[] at (15.5,2.5) {$\lfloor \frac{1}{3}t\rfloor $};
\end{tikzpicture}
}
&
{}
\hackcenter{
\begin{tikzpicture}[scale=0.58]
\draw[  thick, fill=lightgray!70]  (5,0)--(14,0)--(14,3)--(5,3)--(5,0); 
\draw[ thick, fill=yellow]  (8,0)--(14,0)--(14,3)--(8,3)--(8,0); 
\draw[ thick, fill=pink]  (11,1)--(14,1)--(14,3)--(11,3)--(11,1); 
\draw[thick, dotted] (5,0)--(14,0)--(14,3)--(5,3)--(5,0);
\draw[thick, dotted] (5,1)--(5.8,1);
\draw[thick, dotted] (5,2)--(5.8,2);
\draw[thick, dotted] (7.2,1)--(14,1);
\draw[thick, dotted] (7.2,2)--(14,2);
\node at (11.5,1.5) { \textcircled{$\scriptstyle{1}$}};
\node at (8.5,0.5) { \textcircled{$\scriptstyle{2}$}};
\node at (6.5, 1.5) {$\lambda_{\not \succeq \bv}$};
\node[] at (12.5,2.5) {$\frac{1}{2}t$};
\node[] at (12.5,1.5) {$\frac{1}{2}t$};
\node[] at (9.5,2.5) {$\frac{1}{6}(t-2)$};
\node[] at (9.5,1.5) {$\frac{1}{6}(t-2)$};
\node[] at (11,0.5) {$\frac{1}{3}(2t-1)$};
\end{tikzpicture}
}
\\
\end{array}
\end{align*}
\begin{align*}
\begin{array}{c}
\hackcenter{
\begin{tikzpicture}[scale=0.58]
\draw[ thick, fill=lightgray!70]  (2,0)--(17,0)--(17,3)--(2,3)--(2,0); 
\draw[thick, thick,  fill=violet!35]  (5,0)--(17,0)--(17,3)--(5,3)--(5,0); 
\draw[ thick, fill=lime]  (8,1)--(17,1)--(17,3)--(8,3)--(8,1); 
\draw[ thick, fill=pink]  (11,2)--(17,2)--(17,3)--(11,3)--(11,2); 
\draw[ thick, fill=yellow]  (12,1)--(17,1)--(17,2)--(12,2)--(12,1); 
\draw[thick,  fill=cyan!25]  (14,0)--(17,0)--(17,1)--(14,1)--(14,0); 
\draw[thick, dotted] (2,0)--(17,0)--(17,3)--(2,3)--(2,0);
\draw[thick, dotted] (4.2,1)--(17,1);
\draw[thick, dotted] (4.2,2)--(17,2);
\draw[thick, dotted] (2,1)--(2.8,1);
\draw[thick, dotted] (2,2)--(2.8,2);
\node at (11.5,2.5) { \textcircled{$\scriptstyle{1}$}};
\draw[] (11,2)--(11,3);
\node at (12.5,1.5) { \textcircled{$\scriptstyle{2}$}};
\draw[] (12,1)--(12,2);
\node at (14.5,0.5) { \textcircled{$\scriptstyle{4}$}};
\draw[] (14,0)--(14,1);
\node at (8.5,1.5) { \textcircled{$\scriptstyle{3}$}};
\draw[] (8,1)--(8,3);
\node at (5.5,0.5) { \textcircled{$\scriptstyle{5}$}};
\draw[] (5,0)--(5,3);
\node at (3.5, 1.5) {$\lambda_{\not \succeq \bv}$};
\node[] at (14,2.5) {$t$};
\node[] at (15.8,0.5) {$t-3$};
\node[] at (14.5,1.5) {$t-1$};
\node[] at (9.5,2.5) {$\frac{1}{2}(t-3)$};
\node[] at (9.5,0.5) {$\frac{1}{3}(2t-1)$};
\node[] at (10.2,1.5) {$\frac{1}{2}(t-1)$};
\node[] at (6.5,2.5) {$\frac{1}{6}(t-11)$};
\node[] at (6.5,1.5) {$\frac{1}{6}(t-11)$};
\end{tikzpicture}
}
\\
11 \leq t \equiv 5 \pmod 6
\end{array}
\end{align*}
\\

\noindent({\tt Step} 5,  \(\lambda = \llbracket 4, n\rrbracket)\). Define \(t := \tau_2(4n)\). Define \(\bv\) to be the element sequence  in \(\lambda\) depicted below which corresponds to the appropriate condition on \(t\). Go to ({\tt Step} 8).
\begin{align*}
\begin{array}{cc}
\hackcenter{
\begin{tikzpicture}[scale=0.58]
\draw[ thick, fill=lightgray!70]  (2,0)--(5,0)--(5,4)--(2,4)--(2,0); 
\draw[ thick, fill=pink]  (2,3)--(5,3)--(5,4)--(2,4)--(2,3); 
\draw[thick, dotted]  (2,0)--(5,0)--(5,4)--(2,4)--(2,0); 
\draw[thick, dotted] (2,2)--(2.8,2);
\draw[thick, dotted] (2,1)--(2.8,1);
\draw[thick, dotted] (4.2,2)--(5,2);
\draw[thick, dotted] (4.2,1)--(5,1);
\draw[thick, dotted] (2,3)--(5,3);
\node at (2.5,3.5){ \textcircled{$\scriptstyle{1}$}};
\node at (3.5,1.5){ $\lambda_{\not \succeq \bv}$};
\node[] at (3.5,3.5) {$n$};
\end{tikzpicture}
}
&
\hackcenter{
\begin{tikzpicture}[scale=0.58]
\draw[  thick, fill=lightgray!70]  (-1,0)--(17,0)--(17,4)--(-1,4)--(-1,0); 
\draw[ thick,  fill=blue!40!green!45]  (4,0)--(17,0)--(17,4)--(4,4)--(4,0); 
\draw[  thick,  fill=violet!35]  (2,1)--(17,1)--(17,4)--(2,4)--(2,1); 
\draw[thick,  fill=lime]  (2,2)--(17,2)--(17,4)--(2,4)--(2,2);
\draw[ thick, fill=pink]  (6,3)--(17,3)--(17,4)--(6,4)--(6,3); 
\draw[ thick, fill=yellow]  (7,2)--(17,2)--(17,3)--(7,3)--(7,2); 
\draw[thick,  fill=cyan!25]  (9,1)--(17,1)--(17,2)--(9,2)--(9,1); 
\draw[ thick,  fill=orange!65 ]  (11,0)--(17,0)--(17,1)--(11,1)--(11,0); 
\draw[ thick, fill=brown!35]  (2,0)--(6,0)--(6,1)--(2,1)--(2,0); 
\draw[thick, dotted] (-1,0)--(17,0)--(17,4)--(-1,4)--(-1,0);
\draw[thick, dotted] (-1,1)--(17,1);
\draw[thick, dotted] (-1,2)--(-0.2,2);
\draw[thick, dotted] (1.2,2)--(17,2);
\draw[thick, dotted] (-1,3)--(17,3);
\node at (6.5,3.5)  { \textcircled{$\scriptstyle{1}$}};
\draw[] (6,3)--(6,4);
\node at (2.5,2.5)  { \textcircled{$\scriptstyle{3}$}};
\draw[] (2,2)--(2,4);
\node at (2.5,1.5)  { \textcircled{$\scriptstyle{5}$}};
\draw[] (2,1)--(2,4);
\node at (7.5,2.5)  { \textcircled{$\scriptstyle{2}$}};
\draw[] (7,2)--(7,3);
\node at (9.5,1.5)  { \textcircled{$\scriptstyle{4}$}};
\draw[] (9,1)--(9,2);
\node at (11.5,0.5)  { \textcircled{$\scriptstyle{6}$}};
\draw[] (11,0)--(11,1);
\node at (6.5,0.5)  { \textcircled{$\scriptstyle{7}$}};
\draw[] (6,0)--(6,1);
\node at (2.5,0.5)  { \textcircled{$\scriptstyle{8}$}};
\draw[] (2,0)--(2,1);
\node at (0.5,2){$\lambda_{\not \succeq \bv}$};
\node[] at (11.5,3.5) {$11$};
\node[] at (4,3.5) {$4$};
\node[] at (4.5,2.5) {$5$};
\node[] at (12,2.5) {$10$};
\node[] at (14,0.5) {$6$};
\node[] at (8.5,0.5) {$5$};
\node[] at (13,1.5) {$8$};
\node[] at (5.5,1.5) {$7$};
\node[] at (4,0.5) {$4$};
\end{tikzpicture}
}
\\
\end{array}
\end{align*}
\begin{align*}
\begin{array}{cc}
\hackcenter{
\begin{tikzpicture}[scale=0.58]
%
\draw[thick, fill=lightgray!70]  (11,0)--(17,0)--(17,4)--(11,4)--(11,0); 
\draw[ thick, fill=pink]  (14,0)--(17,0)--(17,4)--(14,4)--(14,0); 
\draw[thick, dotted] (11,0)--(17,0)--(17,4)--(11,4)--(11,0);
\draw[thick, dotted] (11,1)--(17,1);
\draw[thick, dotted] (11,2)--(11.8,2);
\draw[thick, dotted] (13.2,2)--(17,2);
\draw[thick, dotted] (11,3)--(17,3);
\node at (14.5,0.5){ \textcircled{$\scriptstyle{1}$}};
\draw[] (14,0)--(14,3);
\node at (12.5,2){$\lambda_{\not \succeq \bv}$};
\node[] at (15.5,0.5) {$\lfloor \frac{1}{4}t\rfloor $};
\node[] at (15.5,1.5) {$\lfloor \frac{1}{4}t\rfloor $};
\node[] at (15.5,2.5) {$\lfloor \frac{1}{4}t\rfloor $};
\node[] at (15.5,3.5) {$\lfloor \frac{1}{4}t\rfloor $};
\end{tikzpicture}
}
&
\hackcenter{
\begin{tikzpicture}[scale=0.58]
\draw[  thick, fill=lightgray!70]  (4,0)--(17,0)--(17,4)--(4,4)--(4,0); 
\draw[thick,  fill=violet!35]  (7,0)--(17,0)--(17,4)--(7,4)--(7,0); 
\draw[ thick, fill=pink]  (10,3)--(17,3)--(17,4)--(10,4)--(10,3); 
\draw[ thick, fill=yellow]  (11,2)--(17,2)--(17,3)--(11,3)--(11,2); 
\draw[thick,  fill=lime]  (12,1)--(17,1)--(17,2)--(12,2)--(12,1); 
\draw[ thick,  fill=cyan!25]  (13,0)--(17,0)--(17,1)--(13,1)--(13,0); 
\draw[thick, dotted] (4,0)--(17,0)--(17,4)--(4,4)--(4,0);
\draw[thick, dotted] (4,1)--(17,1);
\draw[thick, dotted] (4,2)--(4.8,2);
\draw[thick, dotted] (6.2,2)--(17,2);
\draw[thick, dotted] (4,3)--(17,3);
\node at (10.5,3.5) { \textcircled{$\scriptstyle{1}$}};
\draw[] (10,3)--(10,4);
\node at (11.5,2.5) { \textcircled{$\scriptstyle{2}$}};
\draw[] (11,2)--(11,3);
\node at (12.5,1.5) { \textcircled{$\scriptstyle{3}$}};
\draw[] (12,1)--(12,2);
\node at (13.5,0.5) { \textcircled{$\scriptstyle{4}$}};
\draw[] (13,0)--(13,1);
\node at (7.5,0.5) { \textcircled{$\scriptstyle{5}$}};
\draw[] (7,0)--(7,3);
\node at (5.5,2){ $\lambda_{\not \succeq \bv}$};
\node[] at (8.5,3.5) {$\frac{1}{4}(t-10)$};
\node[] at (9,2.5) {$\frac{1}{4}(t-6)$};
\node[] at (9.5,1.5) {$\frac{1}{4}(t-2)$};
\node[] at (10,0.5) {$\frac{1}{4}(t+2)$};
\node[] at (13.5,3.5) {$t$};
\node[] at (14,2.5) {$t-1$};
\node[] at (15,0.5) {$t-3$};
\node[] at (14.5,1.5) {$t-2$};
\end{tikzpicture}
}
\\
\end{array}
\end{align*}
\begin{align*}
\begin{array}{c}
\hackcenter{
\begin{tikzpicture}[scale=0.58]
\draw[ thick, fill=lightgray!70]  (-1,0)--(20,0)--(20,4)--(-1,4)--(-1,0); 
\draw[thick, fill=brown!35]  (2,0)--(20,0)--(20,4)--(2,4)--(2,0); 
\draw[ thick,  fill=orange!60]  (5,2)--(20,2)--(20,4)--(5,4)--(5,2); 
\draw[thick,  fill=blue!40!green!45]  (7,0)--(20,0)--(20,2)--(7,2)--(7,0); 
\draw[thick,  fill=cyan!25]  (10,1)--(20,1)--(20,4)--(10,4)--(10,1);
\draw[thick,  fill=pink]  (13,3)--(20,3)--(20,4)--(13,4)--(13,3); 
\draw[ thick, fill=yellow]  (14,2)--(20,2)--(20,3)--(14,3)--(14,2); 
\draw[thick,  fill=lime]  (15,1)--(20,1)--(20,2)--(15,2)--(15,1); 
\draw[ thick,  fill=violet!35]  (17,0)--(20,0)--(20,1)--(17,1)--(17,0); 
\draw[dotted, thick] (-1,0)--(20,0)--(20,4)--(-1,4)--(-1,0);
\draw[dotted, thick] (-1,1)--(20,1);
\draw[dotted, thick] (-1,2)--(-0.2,2);
\draw[dotted, thick] (1.2,2)--(20,2);
\draw[dotted, thick] (-1,3)--(20,3);
\node at (13.5,3.5) { \textcircled{$\scriptstyle{1}$}};
\node at (10.5,1.5) { \textcircled{$\scriptstyle{4}$}};
\node at (7.5,.5) { \textcircled{$\scriptstyle{7}$}};
\node at (5.5,2.5) { \textcircled{$\scriptstyle{6}$}};
\node at (14.5,2.5) { \textcircled{$\scriptstyle{2}$}};
\node at (15.5,1.5) { \textcircled{$\scriptstyle{3}$}};
\node at (17.5,0.5) { \textcircled{$\scriptstyle{5}$}};
\node at (2.5,0.5) { \textcircled{$\scriptstyle{8}$}};
\node at (0.5,2){ $\lambda_{\not \succeq \bv}$};
\node[] at (16.5,3.5) {$t$};
\node[] at (11.5,3.5) {$\frac{1}{3}(t-6)$};
\node[] at (12,2.5) {$\frac{1}{3}(t-3)$};
\node[] at (12.5,1.5) {$\frac{1}{3}t$};
\node[] at (7.5,3.5) {$\frac{1}{2}(t-5)$};
\node[] at (7.5,2.5) {$\frac{1}{2}(t-5)$};
\node[] at (17,2.5) {$t-1$};
\node[] at (18.7,0.5) {$t-4$};
\node[] at (17.5,1.5) {$t-2$};
\node[] at (12,0.5) {$\frac{1}{3}(2t-6)$};
\node[] at (8.5,1.5) {$\frac{1}{3}(t-12)$};
\node[] at (4.5,0.5) {$\frac{1}{3}(t-3)$};
\node[] at (4.5,1.5) {$\frac{1}{3}(t-3)$};
\node[] at (3.5,3.5) {$\frac{1}{6}(t-15)$};
\node[] at (3.5,2.5) {$\frac{1}{6}(t-15)$};
\end{tikzpicture}
}\\
\end{array}
\end{align*}
\begin{align*}
\begin{array}{c}
\hackcenter{
\begin{tikzpicture}[scale=0.58]
\draw[ thick, fill=lightgray!70]  (-1,0)--(20,0)--(20,4)--(-1,4)--(-1,0); 
\draw[thick, fill=brown!35]  (2,0)--(20,0)--(20,4)--(2,4)--(2,0); 
\draw[  thick,  fill=orange!60]  (5,2)--(20,2)--(20,4)--(5,4)--(5,2); 
\draw[ thick,  fill=blue!40!green!45]  (7,0)--(20,0)--(20,2)--(7,2)--(7,0); 
\draw[ thick,  fill=violet!35]  (10,1)--(20,1)--(20,4)--(10,4)--(10,1);
\draw[ thick, fill=pink]  (13,3)--(20,3)--(20,4)--(13,4)--(13,3); 
\draw[ thick, fill=yellow]  (14,2)--(20,2)--(20,3)--(14,3)--(14,2); 
\draw[ thick, fill=lime]  (15,1)--(20,1)--(20,2)--(15,2)--(15,1); 
\draw[ thick,  fill=cyan!25]  (16,0)--(20,0)--(20,1)--(16,1)--(16,0); 
\draw[thick, dotted] (-1,0)--(20,0)--(20,4)--(-1,4)--(-1,0);
\draw[thick, dotted] (-1,1)--(20,1);
\draw[thick, dotted] (-1,2)--(-0.2,2);
\draw[thick, dotted] (1.2,2)--(20,2);
\draw[thick, dotted] (-1,3)--(20,3);
\node at (13.5,3.5) { \textcircled{$\scriptstyle{1}$}};
\node at (10.5,1.5) { \textcircled{$\scriptstyle{5}$}};
\node at (7.5,.5){ \textcircled{$\scriptstyle{7}$}};
\node at (5.5,2.5) { \textcircled{$\scriptstyle{6}$}};
\node at (14.5,2.5) { \textcircled{$\scriptstyle{2}$}};
\node at (15.5,1.5){ \textcircled{$\scriptstyle{3}$}};
\node at (16.5,0.5) { \textcircled{$\scriptstyle{4}$}};
\node at (2.5,0.5) { \textcircled{$\scriptstyle{8}$}};
\node[] at (12,0.5) {$\frac{1}{3}(2t-8)$};
\node at (0.5,2){ $\lambda_{\not \succeq \bv}$};
\node[] at (16.5,3.5) {$t$};
\node[] at (11.5,3.5) {$\frac{1}{3}(t-7)$};
\node[] at (12,2.5) {$\frac{1}{3}(t-4)$};
\node[] at (12.5,1.5) {$\frac{1}{3}(t-1)$};
\node[] at (7.5,3.5) {$\frac{1}{2}(t-5)$};
\node[] at (7.5,2.5) {$\frac{1}{2}(t-5)$};
\node[] at (17,2.5) {$t-1$};
\node[] at (18,0.5) {$t-3$};
\node[] at (17.5,1.5) {$t-2$};
\node[] at (8.5,1.5) {$\frac{1}{3}(t-10)$};
\node[] at (3.5,3.5) {$\frac{1}{6}(t-13)$};
\node[] at (3.5,2.5) {$\frac{1}{6}(t-13)$};
\node[] at (4.5,0.5) {$\frac{1}{3}(t-4)$};
\node[] at (4.5,1.5) {$\frac{1}{3}(t-4)$};
\end{tikzpicture}
}
\\
\end{array}
\end{align*}
\begin{align*}
\begin{array}{c}
\hackcenter{
\begin{tikzpicture}[scale=0.58]
\draw[ thick, fill=lightgray!70]  (-1,0)--(20,0)--(20,4)--(-1,4)--(-1,0); 
\draw[ thick,  fill=blue!40!green!45]  (2,0)--(20,0)--(20,4)--(2,4)--(2,0); 
\draw[ thick,  fill=violet!35]  (5,1)--(20,1)--(20,4)--(5,4)--(5,1); 
\draw[ thick,  fill=lime]  (8,2)--(20,2)--(20,4)--(8,4)--(8,2);
\draw[ thick, fill=pink]  (12,3)--(20,3)--(20,4)--(12,4)--(12,3); 
\draw[ thick, fill=yellow]  (13,2)--(20,2)--(20,3)--(13,3)--(13,2); 
\draw[ thick,  fill=cyan!25]  (15,1)--(20,1)--(20,2)--(15,2)--(15,1); 
\draw[ thick,  fill=orange!60]  (17,0)--(20,0)--(20,1)--(17,1)--(17,0); 
\draw[thick, dotted] (-1,0)--(20,0)--(20,4)--(-1,4)--(-1,0);
\draw[thick, dotted] (-1,1)--(20,1);
\draw[thick, dotted] (-1,2)--(-0.2,2);
\draw[thick, dotted] (1.2,2)--(20,2);
\draw[thick, dotted] (-1,3)--(20,3);
\node at (12.5,3.5)  { \textcircled{$\scriptstyle{1}$}};
\node at (8.5,2.5)  { \textcircled{$\scriptstyle{3}$}};
\node at (5.5,1.5)  { \textcircled{$\scriptstyle{5}$}};
\node at (13.5,2.5)  { \textcircled{$\scriptstyle{2}$}};
\node at (15.5,1.5)  { \textcircled{$\scriptstyle{4}$}};
\node at (17.5,0.5)  { \textcircled{$\scriptstyle{6}$}};
\node at (2.5,0.5)  { \textcircled{$\scriptstyle{7}$}};
\node at (0.5,2){ $\lambda_{\not \succeq \bv}$};
\node[] at (16,3.5) {$t$};
\node[] at (10,3.5) {$\frac{1}{2}(t-3)$};
\node[] at (10.5,2.5) {$\frac{1}{2}(t-1)$};
\node[] at (6.5,3.5) {$\frac{1}{6}(t-11)$};
\node[] at (6.5,2.5) {$\frac{1}{6}(t-11)$};
\node[] at (10,1.5) {$\frac{1}{3}(2t-1)$};
\node[] at (16.5,2.5) {$t-1$};
\node[] at (18.8,0.5) {$t-5$};
\node[] at (17.5,1.5) {$t-3$};
\node[] at (3.5,3.5) {$\frac{1}{12}(t-23)$};
\node[] at (3.5,2.5) {$\frac{1}{12}(t-23)$};
\node[] at (3.5,1.5) {$\frac{1}{12}(t-23)$};
\node[] at (9.5,0.5) {$\frac{1}{4}(3t-1)$};
\end{tikzpicture}
}
\\
\end{array}
\end{align*}
\\

\noindent({\tt Step} 6,  \(\lambda = \llbracket 5,n \rrbracket)\). Define \(t := \tau_2(5n)\). Define \(\bv\) to be the element sequence  in \(\lambda\) depicted below which corresponds to the appropriate condition on \(t\). Go to ({\tt Step} 8).

\begin{align*}
\begin{array}{ccc}
\hackcenter{
\begin{tikzpicture}[scale=0.58]
\draw[ thick, fill=lightgray!70]  (2,0)--(5,0)--(5,5)--(2,5)--(2,0); 
\draw[ thick, fill=pink]  (2,4)--(5,4)--(5,5)--(2,5)--(2,4); 
\draw[thick, dotted]  (2,0)--(5,0)--(5,4)--(2,4)--(2,0); 
\draw[thick, dotted] (2,2)--(2.8,2);
\draw[thick, dotted] (4.2,2)--(5,2);
\draw[thick, dotted] (2,1)--(5,1);
\draw[thick, dotted] (2,3)--(5,3);
\node at (2.5,4.5){ \textcircled{$\scriptstyle{1}$}};
\node at (3.5,2){$\lambda_{\not \succeq \bv}$};
\node[] at (3.5,4.5) {$n$};
\end{tikzpicture}
}
&
\hackcenter{
\begin{tikzpicture}[scale=0.58]
%
\draw[thick, fill=lightgray!70]  (11,0)--(23,0)--(23,5)--(11,5)--(11,0); 
\draw[ thick,  fill=violet!35]  (14,0)--(23,0)--(23,5)--(14,5)--(14,0); 
\draw[thick,  fill=cyan!25]  (15,1)--(23,1)--(23,5)--(15,5)--(15,1); 
\draw[ thick, fill=lime]  (20,0)--(23,0)--(23,5)--(20,5)--(20,0); 
\draw[ thick, fill=yellow]  (17,2)--(23,2)--(23,5)--(17,5)--(17,2); 
\draw[ thick, fill=pink]  (19,3)--(23,3)--(23,5)--(19,5)--(19,3); 
\draw[thick, dotted] (11,0)--(23,0)--(23,5)--(11,5)--(11,0);
\draw[thick, dotted] (11,1)--(23,1);
\draw[thick, dotted] (11,2)--(11.8,2);
\draw[thick, dotted] (13.2,2)--(23,2);
\draw[thick, dotted] (11,3)--(11.8,3);
\draw[thick, dotted] (13.2,3)--(23,3);
\draw[thick, dotted] (11,4)--(23,4);
\node at (14.5,0.5){ \textcircled{$\scriptstyle{5}$}};
\node at (19.5,3.5){ \textcircled{$\scriptstyle{1}$}};
\node at (17.5,2.5){ \textcircled{$\scriptstyle{2}$}};
\node at (15.5,1.5){ \textcircled{$\scriptstyle{4}$}};
\node at (20.5,0.5){ \textcircled{$\scriptstyle{3}$}};
\draw[] (14,0)--(14,3);
\node at (12.5,2.5){ $\lambda_{\not \succeq \bv}$};
\node[] at (18,4.5) {$2$};
\node[] at (18,3.5) {$2$};
\node[] at (20,2.5) {$9$};
\node[] at (17.5,1.5) {$5$};
\node[] at (17,0.5) {$6$};
\node[] at (16,4.5) {$2$};
\node[] at (16,3.5) {$2$};
\node[] at (16,2.5) {$2$};
\node[] at (14.5,4.5) {$1$};
\node[] at (14.5,3.5) {$1$};
\node[] at (14.5,2.5) {$1$};
\node[] at (14.5,1.5) {$1$};
\node[] at (21,4.5) {$7$};
\node[] at (21,3.5) {$7$};
\node[] at (21.5,1.5) {$6$};
\node[] at (21.5,0.5) {$6$};
\end{tikzpicture}
}
&
\hackcenter{
\begin{tikzpicture}[scale=0.58]
%
\draw[thick, fill=lightgray!70]  (11,0)--(17,0)--(17,5)--(11,5)--(11,0); 
\draw[ thick, fill=pink]  (14,0)--(17,0)--(17,5)--(14,5)--(14,0); 
\draw[thick, dotted] (11,0)--(17,0)--(17,4)--(11,4)--(11,0);
\draw[thick, dotted] (11,1)--(17,1);
\draw[thick, dotted] (11,2)--(11.8,2);
\draw[thick, dotted] (13.2,2)--(17,2);
\draw[thick, dotted] (11,3)--(11.8,3);
\draw[thick, dotted] (13.2,3)--(17,3);
\draw[thick, dotted] (11,4)--(17,4);
\node at (14.5,0.5){ \textcircled{$\scriptstyle{1}$}};
\draw[] (14,0)--(14,3);
\node at (12.5,2.5){$\lambda_{\not \succeq \bv}$};
\node[] at (15.5,0.5) {$\lfloor \frac{1}{5}t\rfloor $};
\node[] at (15.5,1.5) {$\lfloor \frac{1}{5}t\rfloor $};
\node[] at (15.5,2.5) {$\lfloor \frac{1}{5}t\rfloor $};
\node[] at (15.5,3.5) {$\lfloor \frac{1}{5}t\rfloor $};
\node[] at (15.5,4.5) {$\lfloor \frac{1}{5}t\rfloor $};
\end{tikzpicture}
}
\\
\end{array}
\end{align*}

\begin{align*}
\begin{array}{cc}
\hackcenter{
\begin{tikzpicture}[scale=0.58]
%
\draw[thick, fill=lightgray!70]  (11,0)--(23,0)--(23,5)--(11,5)--(11,0); 
\draw[ thick, fill=lime]  (14,0)--(23,0)--(23,5)--(14,5)--(14,0); 
\draw[ thick, fill=yellow]  (18,1)--(23,1)--(23,5)--(18,5)--(18,1); 
\draw[ thick, fill=pink]  (18,3)--(23,3)--(23,5)--(18,5)--(18,3); 
\draw[thick, dotted] (11,0)--(23,0)--(23,5)--(11,5)--(11,0);
\draw[thick, dotted] (11,1)--(23,1);
\draw[thick, dotted] (11,2)--(11.8,2);
\draw[thick, dotted] (13.2,2)--(23,2);
\draw[thick, dotted] (11,3)--(11.8,3);
\draw[thick, dotted] (13.2,3)--(23,3);
\draw[thick, dotted] (11,4)--(23,4);
\node at (14.5,0.5){ \textcircled{$\scriptstyle{3}$}};
\node at (18.5,3.5){ \textcircled{$\scriptstyle{1}$}};
\node at (18.5,1.5){ \textcircled{$\scriptstyle{2}$}};
\draw[] (14,0)--(14,3);
\node at (12.5,2.5){$\lambda_{\not \succeq \bv}$};
\node[] at (16,4.5) {$\frac{1}{10}(t-3)$};
\node[] at (16,3.5) {$\frac{1}{10}(t-3)$};
\node[] at (16,2.5) {$\frac{1}{10}(t-3)$};
\node[] at (16,1.5) {$\frac{1}{10}(t-3)$};
\node[] at (18.5,0.5) {$\frac{1}{5}(3t-4)$};
\node[] at (20.5,4.5) {$\frac{1}{2}(t-1)$};
\node[] at (20.5,3.5) {$\frac{1}{2}(t-1)$};
\node[] at (20.5,2.5) {$\frac{1}{2}(t-1)$};
\node[] at (20.5,1.5) {$\frac{1}{2}(t-1)$};
\end{tikzpicture}
}
&
\hackcenter{
\begin{tikzpicture}[scale=0.58]
%
\draw[thick, fill=lightgray!70]  (11,0)--(23,0)--(23,5)--(11,5)--(11,0); 
\draw[ thick, fill=lime]  (14,0)--(23,0)--(23,5)--(14,5)--(14,0); 
\draw[ thick, fill=yellow]  (19,1)--(23,1)--(23,5)--(19,5)--(19,1); 
\draw[ thick, fill=pink]  (18,3)--(23,3)--(23,5)--(18,5)--(18,3); 
\draw[thick, dotted] (11,0)--(23,0)--(23,5)--(11,5)--(11,0);
\draw[thick, dotted] (11,1)--(23,1);
\draw[thick, dotted] (11,2)--(11.8,2);
\draw[thick, dotted] (13.2,2)--(23,2);
\draw[thick, dotted] (11,3)--(11.8,3);
\draw[thick, dotted] (13.2,3)--(23,3);
\draw[thick, dotted] (11,4)--(23,4);
\node at (14.5,0.5){ \textcircled{$\scriptstyle{3}$}};
\node at (18.5,3.5){ \textcircled{$\scriptstyle{1}$}};
\node at (19.5,1.5){ \textcircled{$\scriptstyle{2}$}};
\draw[] (14,0)--(14,3);
\node at (12.5,2.5){ $\lambda_{\not \succeq \bv}$};
\node[] at (16,4.5) {$\frac{1}{10}(t-8)$};
\node[] at (16,3.5) {$\frac{1}{10}(t-8)$};
\node[] at (16.5,2.5) {$\frac{1}{10}(t+2)$};
\node[] at (16.5,1.5) {$\frac{1}{10}(t+2)$};
\node[] at (18.5,0.5) {$\frac{1}{5}(3t-4)$};
\node[] at (20.5,4.5) {$\frac{1}{2}t$};
\node[] at (20.5,3.5) {$\frac{1}{2}t$};
\node[] at (21,2.5) {$\frac{1}{2}(t-2)$};
\node[] at (21,1.5) {$\frac{1}{2}(t-2)$};
\end{tikzpicture}
}\\
\end{array}
\end{align*}
\begin{align*}
\begin{array}{c}
\hackcenter{
\begin{tikzpicture}[scale=0.58]
\draw[ thick, fill=lightgray!70]  (-4,0)--(17,0)--(17,5)--(-4,5)--(-4,0); 
\draw[ thick, fill=pink!70!blue!30]  (-1,0)--(17,0)--(17,5)--(-1,5)--(-1,0); 
\draw[thick, fill=brown!35]  (8,0)--(17,0)--(17,5)--(8,5)--(8,0); 
\draw[ thick,  fill=blue!40!green!45]  (3,2)--(17,2)--(17,5)--(3,5)--(3,2); 
\draw[  thick,  fill=orange!60]  (6,3)--(17,3)--(17,5)--(6,5)--(6,3);
\draw[thick,  fill=violet!35]  (14,0)--(17,0)--(17,5)--(14,5)--(14,0); 
\draw[thick,  fill=cyan!25]  (13,1)--(17,1)--(17,5)--(13,5)--(13,1); 
\draw[ thick, fill=lime]  (12,2)--(17,2)--(17,5)--(12,5)--(12,2); 
\draw[thick, fill=yellow]  (11,3)--(17,3)--(17,5)--(11,5)--(11,3); 
\draw[thick, fill=pink]  (10,4)--(17,4)--(17,5)--(10,5)--(10,4); 
\draw[thick, dotted] (-4,0)--(17,0)--(17,5)--(-4,5)--(-4,0);
\draw[thick, dotted] (-4,1)--(17,1);
\draw[thick, dotted] (-4,2)--(-3.2,2);
\draw[thick, dotted] (-1.8,2)--(17,2);
\draw[thick, dotted] (-4,3)--(-3.2,3);
\draw[thick, dotted] (-1.8,3)--(17,3);
\draw[thick, dotted] (-4,4)--(17,4);
\node at (10.5,4.5)  { \textcircled{$\scriptstyle{1}$}};
\node at (3.5,2.5)  { \textcircled{$\scriptstyle{7}$}};
\node at (6.5,3.5)  { \textcircled{$\scriptstyle{6}$}};
\node at (13.5,1.5)  { \textcircled{$\scriptstyle{4}$}};
\node at (11.5,3.5)  { \textcircled{$\scriptstyle{2}$}};
\node at (12.5,2.5)  { \textcircled{$\scriptstyle{3}$}};
\node at (14.5,0.5)  { \textcircled{$\scriptstyle{5}$}};
\node at (8.5,0.5)  { \textcircled{$\scriptstyle{8}$}};
\node at (-0.5,0.5)  { \textcircled{$\scriptstyle{9}$}};
\node at (-2.5,2.5){ $\lambda_{\not \succeq \bv}$};
%
\node[] at (13.5,4.5) {$t$};
\node[] at (8,4.5) {$\frac{1}{2}(t-6)$};
\node[] at (8.5,3.5) {$\frac{1}{2}(t-4)$};
\node[] at (4.5,4.5) {$\frac{1}{6}(t-10)$};
\node[] at (4.5,3.5) {$\frac{1}{6}(t-10)$};
\node[] at (7.5,2.5) {$\frac{1}{3}(2t-8)$};
\node[] at (1,4.5) {$\frac{1}{15}(2t-38)$};
\node[] at (1,3.5) {$\frac{1}{15}(2t-38)$};
\node[] at (1,2.5) {$\frac{1}{15}(2t-38)$};
\node[] at (3.5,1.5) {$\frac{1}{10}(3t-2)$};
\node[] at (3.5,0.5) {$\frac{1}{10}(3t-2)$};
\node[] at (14,3.5) {$t-1$};
\node[] at (15.7,0.5) {$t-4$};
\node[] at (10.5,1.5) {$\frac{1}{2}(t-8)$};
\node[] at (11,0.5) {$\frac{1}{2}(t-6)$};
\node[] at (15,1.5) {$t-3$};
\node[] at (14.5,2.5) {$t-2$};
\end{tikzpicture}
}
\\
\end{array}
\end{align*}
\begin{align*}
\begin{array}{c}
\hackcenter{
\begin{tikzpicture}[scale=0.58]
\draw[ thick, fill=lightgray!70]  (-4,0)--(17,0)--(17,5)--(-4,5)--(-4,0); 
\draw[ thick, fill=pink!70!blue!30]  (-1,0)--(17,0)--(17,5)--(-1,5)--(-1,0); 
\draw[thick, fill=brown!35]  (5,0)--(17,0)--(17,5)--(5,5)--(5,0); 
\draw[   thick,  fill=blue!40!green!45  ]  (4,1)--(17,1)--(17,5)--(4,5)--(4,1); 
\draw[ thick,  fill=orange!60]  (3,2)--(17,2)--(17,5)--(3,5)--(3,2); 
\draw[thick,  fill=violet!35  ]  (6,3)--(17,3)--(17,5)--(6,5)--(6,3);
\draw[ thick,  fill=cyan!25]  (13,0)--(17,0)--(17,5)--(13,5)--(13,0); 
\draw[ thick, fill=lime]  (11,2)--(17,2)--(17,5)--(11,5)--(11,2); 
\draw[thick, fill=yellow]  (10,3)--(17,3)--(17,5)--(10,5)--(10,3); 
\draw[thick, fill=pink]  (9,4)--(17,4)--(17,5)--(9,5)--(9,4); 
\draw[thick, dotted] (-4,0)--(17,0)--(17,5)--(-4,5)--(-4,0);
\draw[thick, dotted] (-4,1)--(17,1);
\draw[thick, dotted] (-4,2)--(-3.2,2);
\draw[thick, dotted] (-1.8,2)--(17,2);
\draw[thick, dotted] (-4,3)--(-3.2,3);
\draw[thick, dotted] (-1.8,3)--(17,3);
\draw[thick, dotted] (-4,4)--(17,4);
\node at (9.5,4.5)  { \textcircled{$\scriptstyle{1}$}};
\node at (3.5,2.5)  { \textcircled{$\scriptstyle{6}$}};
\node at (6.5,3.5)  { \textcircled{$\scriptstyle{5}$}};
\node at (4.5,1.5)  { \textcircled{$\scriptstyle{7}$}};
\node at (10.5,3.5)  { \textcircled{$\scriptstyle{2}$}};
\node at (11.5,2.5)  { \textcircled{$\scriptstyle{3}$}};
\node at (13.5,0.5)  { \textcircled{$\scriptstyle{4}$}};
\node at (5.5,0.5)  { \textcircled{$\scriptstyle{8}$}};
\node at (-0.5,0.5)  { \textcircled{$\scriptstyle{9}$}};
\node at (-2.5,2.5){ $\lambda_{\not \succeq \bv}$};
%
\node[] at (13,4.5) {$t$};
\node[] at (7.5,4.5) {$\frac{1}{2}(t-5)$};
\node[] at (8.2,3.5) {$\frac{1}{2}(t-3)$};
\node[] at (4.5,4.5) {$\frac{1}{6}(t-9)$};
\node[] at (4.5,3.5) {$\frac{1}{6}(t-9)$};
\node[] at (7,2.5) {$\frac{1}{3}(2t-6)$};
\node[] at (1,4.5) {$\frac{1}{15}(2t-48)$};
\node[] at (1,3.5) {$\frac{1}{15}(2t-48)$};
\node[] at (1,2.5) {$\frac{1}{15}(2t-48)$};
\node[] at (1.5,1.5) {$\frac{1}{10}(3t+3)$};
\node[] at (2,0.5) {$\frac{1}{10}(3t+13)$};
\node[] at (13.5,3.5) {$t-1$};
\node[] at (15.2,0.5) {$\frac{1}{2}(t-3)$};
\node[] at (8.5,1.5) {$t-6$};
\node[] at (9,0.5) {$t-7$};
\node[] at (15.2,1.5) {$\frac{1}{2}(t-3)$};
\node[] at (14,2.5) {$t-2$};
\end{tikzpicture}
}
\\
\end{array}
\end{align*}
\begin{align*}
\begin{array}{c}
\hackcenter{
\begin{tikzpicture}[scale=0.58]
\draw[ thick, fill=lightgray!70]  (-4,0)--(17,0)--(17,5)--(-4,5)--(-4,0); 
\draw[ thick, fill=pink!70!blue!30]  (-1,0)--(17,0)--(17,5)--(-1,5)--(-1,0); 
\draw[thick, fill=brown!35]  (8,0)--(17,0)--(17,5)--(8,5)--(8,0); 
\draw[thick,  fill=violet!35]  (3,2)--(17,2)--(17,5)--(3,5)--(3,2); 
\draw[ thick,  fill=cyan!25]  (6,3)--(17,3)--(17,5)--(6,5)--(6,3);
\draw[  thick,  fill=blue!40!green!45 ]  (14,0)--(17,0)--(17,5)--(14,5)--(14,0); 
\draw[thick,  fill=orange!60  ]  (13,1)--(17,1)--(17,5)--(13,5)--(13,1); 
\draw[ thick, fill=lime]  (11,2)--(17,2)--(17,5)--(11,5)--(11,2); 
\draw[thick, fill=yellow]  (10,3)--(17,3)--(17,5)--(10,5)--(10,3); 
\draw[thick, fill=pink]  (9,4)--(17,4)--(17,5)--(9,5)--(9,4); 
\draw[thick, dotted] (-4,0)--(17,0)--(17,5)--(-4,5)--(-4,0);
\draw[thick, dotted] (-4,1)--(17,1);
\draw[thick, dotted] (-4,2)--(-3.2,2);
\draw[thick, dotted] (-1.8,2)--(17,2);
\draw[thick, dotted] (-4,3)--(-3.2,3);
\draw[thick, dotted] (-1.8,3)--(17,3);
\draw[thick, dotted] (-4,4)--(17,4);
\node at (9.5,4.5)  { \textcircled{$\scriptstyle{1}$}};
\node at (3.5,2.5)  { \textcircled{$\scriptstyle{5}$}};
\node at (6.5,3.5)  { \textcircled{$\scriptstyle{4}$}};
\node at (13.5,1.5)  { \textcircled{$\scriptstyle{6}$}};
\node at (10.5,3.5)  { \textcircled{$\scriptstyle{2}$}};
\node at (11.5,2.5)  { \textcircled{$\scriptstyle{3}$}};
\node at (14.5,0.5)  { \textcircled{$\scriptstyle{7}$}};
\node at (8.5,0.5)  { \textcircled{$\scriptstyle{8}$}};
\node at (-0.5,0.5)  { \textcircled{$\scriptstyle{9}$}};
\node at (-2.5,2.5){ $\lambda_{\not \succeq \bv}$};
%
\node[] at (13,4.5) {$t$};
\node[] at (7.5,4.5) {$\frac{1}{2}(t-4)$};
\node[] at (8.2,3.5) {$\frac{1}{2}(t-2)$};
\node[] at (4.5,4.5) {$\frac{1}{6}(t-8)$};
\node[] at (4.5,3.5) {$\frac{1}{6}(t-8)$};
\node[] at (7,2.5) {$\frac{1}{3}(2t-4)$};
\node[] at (1,4.5) {$\frac{1}{15}(2t-58)$};
\node[] at (1,3.5) {$\frac{1}{15}(2t-58)$};
\node[] at (1,2.5) {$\frac{1}{15}(2t-58)$};
\node[] at (3.5,1.5) {$\frac{1}{10}(3t+18)$};
\node[] at (3.5,0.5) {$\frac{1}{10}(3t+18)$};
\node[] at (13.5,3.5) {$t-1$};
\node[] at (15.7,0.5) {$t-6$};
\node[] at (10.5,1.5) {$\frac{1}{2}(t-8)$};
\node[] at (11,0.5) {$\frac{1}{2}(t-6)$};
\node[] at (15,1.5) {$t-5$};
\node[] at (14,2.5) {$t-2$};
\end{tikzpicture}
}
\\
\end{array}
\end{align*}
\begin{align*}
\begin{array}{c}
\hackcenter{
\begin{tikzpicture}[scale=0.58]
\draw[ thick, fill=lightgray!70]  (-4,0)--(17,0)--(17,5)--(-4,5)--(-4,0); 
\draw[ thick, fill=pink!70!blue!30]  (-1,0)--(17,0)--(17,5)--(-1,5)--(-1,0); 
\draw[thick, fill=brown!35 ]  (3,1)--(17,1)--(17,5)--(3,5)--(3,1); 
\draw[ thick,  fill=blue!40!green!45   ]  (8,0)--(17,0)--(17,5)--(8,5)--(8,0); 
\draw[ thick,  fill=violet!35 ]  (6,3)--(17,3)--(17,5)--(6,5)--(6,3);
\draw[  thick,  fill=orange!60 ]  (14,0)--(17,0)--(17,5)--(14,5)--(14,0); 
\draw[thick,  fill=cyan!25 ]  (12,1)--(17,1)--(17,5)--(12,5)--(12,1); 
\draw[ thick, fill=lime]  (11,2)--(17,2)--(17,5)--(11,5)--(11,2); 
\draw[thick, fill=yellow]  (10,3)--(17,3)--(17,5)--(10,5)--(10,3); 
\draw[thick, fill=pink]  (9,4)--(17,4)--(17,5)--(9,5)--(9,4); 
\draw[thick, dotted] (-4,0)--(17,0)--(17,5)--(-4,5)--(-4,0);
\draw[thick, dotted] (-4,1)--(17,1);
\draw[thick, dotted] (-4,2)--(-3.2,2);
\draw[thick, dotted] (-1.8,2)--(17,2);
\draw[thick, dotted] (-4,3)--(-3.2,3);
\draw[thick, dotted] (-1.8,3)--(17,3);
\draw[thick, dotted] (-4,4)--(17,4);
\node at (9.5,4.5)  { \textcircled{$\scriptstyle{1}$}};
\node at (3.5,1.5)  { \textcircled{$\scriptstyle{8}$}};
\node at (6.5,3.5)  { \textcircled{$\scriptstyle{5}$}};
\node at (12.5,1.5)  { \textcircled{$\scriptstyle{4}$}};
\node at (10.5,3.5)  { \textcircled{$\scriptstyle{2}$}};
\node at (11.5,2.5)  { \textcircled{$\scriptstyle{3}$}};
\node at (14.5,0.5)  { \textcircled{$\scriptstyle{6}$}};
\node at (8.5,0.5)  { \textcircled{$\scriptstyle{7}$}};
\node at (-0.5,0.5)  { \textcircled{$\scriptstyle{9}$}};
\node at (-2.5,2.5){ $\lambda_{\not \succeq \bv}$};
%
\node[] at (13,4.5) {$t$};
\node[] at (7.5,4.5) {$\frac{1}{2}(t-5)$};
\node[] at (8.2,3.5) {$\frac{1}{2}(t-3)$};
\node[] at (4.5,4.5) {$\frac{1}{6}(t-19)$};
\node[] at (4.5,3.5) {$\frac{1}{6}(t-19)$};
\node[] at (5.5,2.5) {$\frac{1}{3}(t-1)$};
\node[] at (5.5,1.5) {$\frac{1}{3}(t-1)$};
\node[] at (1,4.5) {$\frac{1}{15}(2t-23)$};
\node[] at (1,3.5) {$\frac{1}{15}(2t-23)$};
\node[] at (1,2.5) {$\frac{1}{15}(2t-23)$};
\node[] at (1,1.5) {$\frac{1}{15}(2t-23)$};
\node[] at (3.5,0.5) {$\frac{1}{15}(7t-28)$};
\node[] at (13.5,3.5) {$t-1$};
\node[] at (15.7,0.5) {$t-5$};
\node[] at (9.5,2.5) {$\frac{1}{3}(t-10)$};
\node[] at (10,1.5) {$\frac{1}{3}(t-7)$};
\node[] at (11,0.5) {$\frac{1}{3}(t-1)$};
\node[] at (15,1.5) {$t-3$};
\node[] at (14,2.5) {$t-2$};
\end{tikzpicture}
}
\\
\end{array}
\end{align*}
\begin{align*}
\begin{array}{c}
\hackcenter{
\begin{tikzpicture}[scale=0.58]
\draw[ thick, fill=lightgray!70]  (-4,0)--(17,0)--(17,5)--(-4,5)--(-4,0); 
\draw[ thick, fill=pink!70!blue!30]  (-1,0)--(17,0)--(17,5)--(-1,5)--(-1,0); 
\draw[thick, fill=brown!35]  (8,0)--(17,0)--(17,5)--(8,5)--(8,0); 
\draw[ thick,  fill=blue!40!green!45 ]  (3,2)--(17,2)--(17,5)--(3,5)--(3,2); 
\draw[ thick,  fill=cyan!25]  (6,2)--(17,2)--(17,5)--(6,5)--(6,2);
\draw[ thick,  fill=orange!60   ]  (14,0)--(17,0)--(17,5)--(14,5)--(14,0); 
\draw[thick,  fill=violet!35   ]  (13,1)--(17,1)--(17,5)--(13,5)--(13,1); 
\draw[ thick, fill=lime]  (11,2)--(17,2)--(17,5)--(11,5)--(11,2); 
\draw[thick, fill=yellow]  (10,3)--(17,3)--(17,5)--(10,5)--(10,3); 
\draw[thick, fill=pink]  (9,4)--(17,4)--(17,5)--(9,5)--(9,4); 
\draw[thick, dotted] (-4,0)--(17,0)--(17,5)--(-4,5)--(-4,0);
\draw[thick, dotted] (-4,1)--(17,1);
\draw[thick, dotted] (-4,2)--(-3.2,2);
\draw[thick, dotted] (-1.8,2)--(17,2);
\draw[thick, dotted] (-4,3)--(-3.2,3);
\draw[thick, dotted] (-1.8,3)--(17,3);
\draw[thick, dotted] (-4,4)--(17,4);
\node at (9.5,4.5)  { \textcircled{$\scriptstyle{1}$}};
\node at (3.5,2.5)  { \textcircled{$\scriptstyle{7}$}};
\node at (6.5,2.5)  { \textcircled{$\scriptstyle{4}$}};
\node at (13.5,1.5)  { \textcircled{$\scriptstyle{5}$}};
\node at (10.5,3.5)  { \textcircled{$\scriptstyle{2}$}};
\node at (11.5,2.5)  { \textcircled{$\scriptstyle{3}$}};
\node at (14.5,0.5)  { \textcircled{$\scriptstyle{6}$}};
\node at (8.5,0.5)  { \textcircled{$\scriptstyle{8}$}};
\node at (-0.5,0.5)  { \textcircled{$\scriptstyle{9}$}};
\node at (-2.5,2.5){ $\lambda_{\not \succeq \bv}$};
%
\node[] at (13,4.5) {$t$};
\node[] at (7.5,4.5) {$\frac{1}{3}(t-6)$};
\node[] at (8,3.5) {$\frac{1}{3}(t-3)$};
\node[] at (8.5,2.5) {$\frac{1}{3}t$};
\node[] at (4.5,4.5) {$\frac{1}{3}(t-6)$};
\node[] at (4.5,3.5) {$\frac{1}{3}(t-6)$};
\node[] at (4.9,2.5) {$\frac{1}{3}(t\hspace{-0.6mm}-\hspace{-0.6mm}6)$};
\node[] at (1,4.5) {$\frac{1}{15}(2t-48)$};
\node[] at (1,3.5) {$\frac{1}{15}(2t-48)$};
\node[] at (1,2.5) {$\frac{1}{15}(2t-48)$};
\node[] at (3.5,1.5) {$\frac{1}{10}(3t+8)$};
\node[] at (3.5,0.5) {$\frac{1}{10}(3t+8)$};
\node[] at (13.5,3.5) {$t-1$};
\node[] at (15.7,0.5) {$t-5$};
\node[] at (10.5,1.5) {$\frac{1}{2}(t-8)$};
\node[] at (11,0.5) {$\frac{1}{2}(t-6)$};
\node[] at (15,1.5) {$t-4$};
\node[] at (14,2.5) {$t-2$};
\end{tikzpicture}
}
\\
\end{array}
\end{align*}
\begin{align*}
\begin{array}{c}
\hackcenter{
\begin{tikzpicture}[scale=0.58]
\draw[ thick, fill=lightgray!70]  (-4,0)--(17,0)--(17,5)--(-4,5)--(-4,0); 
\draw[ thick, fill=pink!70!blue!30]  (-1,0)--(17,0)--(17,5)--(-1,5)--(-1,0); 
\draw[thick, fill=brown!35]  (8,0)--(17,0)--(17,5)--(8,5)--(8,0); 
\draw[  thick,  fill=blue!40!green!45  ]  (3,2)--(17,2)--(17,5)--(3,5)--(3,2); 
\draw[thick,  fill=violet!35 ]  (6,3)--(17,3)--(17,5)--(6,5)--(6,3);
\draw[  thick,  fill=orange!60 ]  (14,0)--(17,0)--(17,5)--(14,5)--(14,0); 
\draw[   thick,  fill=cyan!25 ]  (12,1)--(17,1)--(17,5)--(12,5)--(12,1); 
\draw[ thick, fill=lime]  (11,2)--(17,2)--(17,5)--(11,5)--(11,2); 
\draw[thick, fill=yellow]  (10,3)--(17,3)--(17,5)--(10,5)--(10,3); 
\draw[thick, fill=pink]  (9,4)--(17,4)--(17,5)--(9,5)--(9,4); 
\draw[thick, dotted] (-4,0)--(17,0)--(17,5)--(-4,5)--(-4,0);
\draw[thick, dotted] (-4,1)--(17,1);
\draw[thick, dotted] (-4,2)--(-3.2,2);
\draw[thick, dotted] (-1.8,2)--(17,2);
\draw[thick, dotted] (-4,3)--(-3.2,3);
\draw[thick, dotted] (-1.8,3)--(17,3);
\draw[thick, dotted] (-4,4)--(17,4);
\node at (9.5,4.5)  { \textcircled{$\scriptstyle{1}$}};
\node at (3.5,2.5)  { \textcircled{$\scriptstyle{7}$}};
\node at (6.5,3.5)  { \textcircled{$\scriptstyle{5}$}};
\node at (12.5,1.5)  { \textcircled{$\scriptstyle{4}$}};
\node at (10.5,3.5)  { \textcircled{$\scriptstyle{2}$}};
\node at (11.5,2.5)  { \textcircled{$\scriptstyle{3}$}};
\node at (14.5,0.5)  { \textcircled{$\scriptstyle{6}$}};
\node at (8.5,0.5)  { \textcircled{$\scriptstyle{8}$}};
\node at (-0.5,0.5)  { \textcircled{$\scriptstyle{9}$}};
\node at (-2.5,2.5){ $\lambda_{\not \succeq \bv}$};
%
\node[] at (13,4.5) {$t$};
\node[] at (7.5,4.5) {$\frac{1}{2}(t-5)$};
\node[] at (8.2,3.5) {$\frac{1}{2}(t-3)$};
\node[] at (4.5,4.5) {$\frac{1}{6}(t-11)$};
\node[] at (4.5,3.5) {$\frac{1}{6}(t-11)$};
\node[] at (7,2.5) {$\frac{1}{3}(2t-7)$};
\node[] at (1,4.5) {$\frac{1}{15}(2t-43)$};
\node[] at (1,3.5) {$\frac{1}{15}(2t-43)$};
\node[] at (1,2.5) {$\frac{1}{15}(2t-43)$};
\node[] at (3.5,1.5) {$\frac{1}{10}(3t+3)$};
\node[] at (3.5,0.5) {$\frac{1}{10}(3t+3)$};
\node[] at (13.5,3.5) {$t-1$};
\node[] at (15.7,0.5) {$t-5$};
\node[] at (10,1.5) {$\frac{1}{2}(t-9)$};
\node[] at (10.5,0.5) {$\frac{1}{2}(t-5)$};
\node[] at (14.5,1.5) {$t-3$};
\node[] at (14,2.5) {$t-2$};
\end{tikzpicture}
}
\\
\end{array}
\end{align*}
\\


\noindent({\tt Step} 7,  \(\lambda =\llbracket 6,n \rrbracket)\). Define \(t := \tau_2(6n)\). Define \(\bv\) to be the element sequence  in \(\lambda\) depicted below which corresponds to the appropriate conditions on \(n\) and \(t\). Go to ({\tt Step} 8).

\begin{align*}
\begin{array}{cccc}
\hackcenter{
\begin{tikzpicture}[scale=0.58]
%
\draw[thick, fill=brown!35]  (0,0)--(6,0)--(6,6)--(0,6)--(0,0);
\draw[ thick,  fill=blue!40!green!45]  (1,0)--(6,0)--(6,6)--(1,6)--(1,0);
\draw[ thick,  fill=orange!60]  (0,1)--(6,1)--(6,6)--(0,6)--(0,1);
\draw[thick,  fill=violet!35]  (3,0)--(6,0)--(6,6)--(3,6)--(3,0);
\draw[thick,  fill=cyan!25]  (0,2)--(6,2)--(6,6)--(0,6)--(0,2);
\draw[thick, fill=lime]  (3,1)--(6,1)--(6,6)--(3,6)--(3,1);
\draw[thick, fill=yellow]  (1,3)--(6,3)--(6,6)--(1,6)--(1,3);
\draw[thick, fill=pink]  (2,4)--(6,4)--(6,6)--(2,6)--(2,4);
\draw[thick, dotted] (0,0)--(0,6);
\draw[thick, dotted] (1,0)--(1,6);
\draw[thick, dotted] (2,0)--(2,6);
\draw[thick, dotted] (3,0)--(3,6);
\draw[thick, dotted] (4,0)--(4,6);
\draw[thick, dotted] (5,0)--(5,6);
\draw[thick, dotted] (0,1)--(6,1);
\draw[thick, dotted] (0,2)--(6,2);
\draw[thick, dotted] (0,3)--(6,3);
\draw[thick, dotted] (0,4)--(6,4);
\draw[thick, dotted] (0,5)--(6,5);
\node at (2.5,4.5){ \textcircled{$\scriptstyle{1}$}};
\node at (1.5,3.5){ \textcircled{$\scriptstyle{2}$}};
\node at (3.5,1.5){ \textcircled{$\scriptstyle{3}$}};
\node at (0.5,2.5){ \textcircled{$\scriptstyle{4}$}};
\node at (3.5,0.5){ \textcircled{$\scriptstyle{5}$}};
\node at (0.5,1.5){ \textcircled{$\scriptstyle{6}$}};
\node at (1.5,0.5){ \textcircled{$\scriptstyle{7}$}};
\node at (0.5,0.5){ \textcircled{$\scriptstyle{8}$}};
%
%
\end{tikzpicture}
}
&
\hackcenter{
\begin{tikzpicture}[scale=0.58]
\draw[ thick, fill=lightgray!70]  (2,0)--(5,0)--(5,6)--(2,6)--(2,0); 
\draw[ thick, fill=pink]  (2,5)--(5,5)--(5,6)--(2,6)--(2,5); 
\draw[thick, dotted]  (2,0)--(5,0)--(5,6)--(2,6)--(2,0); 
\draw[thick, dotted] (2,2)--(2.8,2);
\draw[thick, dotted] (4.2,2)--(5,2);
\draw[thick, dotted] (2,3)--(2.8,3);
\draw[thick, dotted] (4.2,3)--(5,3);
\draw[thick, dotted] (2,1)--(5,1);
\draw[thick, dotted] (2,4)--(5,4);
\draw[thick, dotted] (2,5)--(5,5);
\node at (2.5,5.5){ \textcircled{$\scriptstyle{1}$}};
\node at (3.5,2.5){ $\lambda_{\not \succeq \bv}$};
\node[] at (3.5,5.5) {$n$};
\end{tikzpicture}
}
&
\hackcenter{
\begin{tikzpicture}[scale=0.58]
%
\draw[thick, fill=lightgray!70]  (11,0)--(17,0)--(17,6)--(11,6)--(11,0); 
\draw[ thick, fill=pink]  (14,0)--(17,0)--(17,6)--(14,6)--(14,0); 
\draw[thick, dotted] (11,0)--(17,0)--(17,6)--(11,6)--(11,0);
\draw[thick, dotted] (11,1)--(17,1);
\draw[thick, dotted] (11,2)--(17,2);
\draw[thick, dotted] (11,3)--(11.8,3);
\draw[thick, dotted] (13.2,3)--(17,3);
\draw[thick, dotted] (11,4)--(17,4);
\draw[thick, dotted] (11,5)--(17,5);
\node at (14.5,0.5){ \textcircled{$\scriptstyle{1}$}};
\draw[] (14,0)--(14,3);
\node at (12.5,3){ $\lambda_{\not \succeq \bv}$};
\node[] at (15.5,0.5) {$\lfloor \frac{1}{6}t\rfloor $};
\node[] at (15.5,1.5) {$\lfloor \frac{1}{6}t\rfloor $};
\node[] at (15.5,2.5) {$\lfloor \frac{1}{6}t\rfloor $};
\node[] at (15.5,3.5) {$\lfloor \frac{1}{6}t\rfloor $};
\node[] at (15.5,4.5) {$\lfloor \frac{1}{6}t\rfloor $};
\node[] at (15.5,5.5) {$\lfloor \frac{1}{6}t\rfloor $};
\end{tikzpicture}
}
&
\hackcenter{
\begin{tikzpicture}[scale=0.58]
%
\draw[thick, fill=lightgray!70]  (11,0)--(17,0)--(17,6)--(11,6)--(11,0); 
\draw[ thick, fill=yellow]  (14,0)--(17,0)--(17,6)--(14,6)--(14,0); 
\draw[ thick, fill=pink]  (14,3)--(17,3)--(17,6)--(14,6)--(14,3); 
\draw[thick, dotted] (11,0)--(17,0)--(17,6)--(11,6)--(11,0);
\draw[thick, dotted] (11,1)--(17,1);
\draw[thick, dotted] (11,2)--(17,2);
\draw[thick, dotted] (11,3)--(11.8,3);
\draw[thick, dotted] (13.2,3)--(17,3);
\draw[thick, dotted] (11,4)--(17,4);
\draw[thick, dotted] (11,5)--(17,5);
\node at (14.5,3.5){ \textcircled{$\scriptstyle{1}$}};
\node at (14.5,0.5){ \textcircled{$\scriptstyle{2}$}};
\draw[] (14,0)--(14,3);
\node at (12.5,3){ $\lambda_{\not \succeq \bv}$};
\node[] at (15.9,0.5) {$ \frac{1}{3}(t\hspace{-0.8mm}-\hspace{-0.8mm}1)$};
\node[] at (15.5,1.5) {$ \frac{1}{3}(t-1)$};
\node[] at (15.5,2.5) {$ \frac{1}{3}(t-1)$};
\node[] at (15.9,3.5) {$ \frac{1}{3}(t\hspace{-0.8mm}-\hspace{-0.8mm}1)$};
\node[] at (15.5,4.5) {$ \frac{1}{3}(t-1)$};
\node[] at (15.5,5.5) {$ \frac{1}{3}(t-1)$};
\end{tikzpicture}
}
\\
\end{array}
\end{align*}
\begin{align*}
\begin{array}{c}
\hackcenter{
\begin{tikzpicture}[scale=0.58]
%
\draw[thick, fill=blue!40!green!45]  (0,0)--(20,0)--(20,6)--(0,6)--(0,0);
\draw[  thick, fill=orange!60]  (1,0)--(20,0)--(20,6)--(1,6)--(1,0);  
\draw[  thick,fill=violet!35]  (3,0)--(20,0)--(20,6)--(3,6)--(3,0);  
\draw[  thick, fill=cyan!25]  (0,1)--(20,1)--(20,6)--(0,6)--(0,1);  
\draw[  thick,  fill=lime]  (6,0)--(20,0)--(20,6)--(6,6)--(6,0);  
\draw[  thick,  fill=yellow]  (0,2)--(20,2)--(20,6)--(0,6)--(0,2);  
\draw[  thick, fill=pink!70!blue!30]  (4,1)--(20,1)--(20,6)--(4,6)--(4,1);  
\draw[  thick,  fill=blue!40!green!45]  (3,2)--(20,2)--(20,6)--(3,6)--(3,2);  
\draw[ thick,  fill=violet!35]  (0,4)--(20,4)--(20,6)--(0,6)--(0,4); 
\draw[ thick, fill=brown!35]  (11,0)--(20,0)--(20,6)--(11,6)--(11,0); 
\draw[ thick,  fill=orange!60]  (15,0)--(20,0)--(20,6)--(15,6)--(15,0); 
\draw[thick,  fill=cyan!25]  (8,2)--(20,2)--(20,6)--(8,6)--(8,2);
\draw[ thick, fill=lime]  (7,3)--(20,3)--(20,6)--(7,6)--(7,3);
\draw[ thick, fill=yellow]  (6,4)--(20,4)--(20,6)--(6,6)--(6,4);
\draw[ thick, fill=pink]  (5,5)--(20,5)--(20,6)--(5,6)--(5,5);
\draw[thick, dotted] (0,0)--(0,6);
\draw[thick, dotted] (1,0)--(1,6);
\draw[thick, dotted] (2,0)--(2,6);
\draw[thick, dotted] (3,0)--(3,6);
\draw[thick, dotted] (4,0)--(4,6);
\draw[thick, dotted] (5,0)--(5,6);
\draw[thick, dotted] (6,0)--(6,6);
\draw[thick, dotted] (7,0)--(7,6);
\draw[thick, dotted] (8,0)--(8,6);
\draw[thick, dotted] (9,0)--(9,6);
\draw[thick, dotted] (10,0)--(10,6);
\draw[thick, dotted] (11,0)--(11,6);
\draw[thick, dotted] (12,0)--(12,6);
\draw[thick, dotted] (13,0)--(13,6);
\draw[thick, dotted] (14,0)--(14,6);
\draw[thick, dotted] (15,0)--(15,6);
\draw[thick, dotted] (16,0)--(16,6);
\draw[thick, dotted] (17,0)--(17,6);
\draw[thick, dotted] (18,0)--(18,6);
\draw[thick, dotted] (19,0)--(19,6);
\draw[thick, dotted] (0,1)--(20,1);
\draw[thick, dotted] (0,3)--(20,3);
\draw[thick, dotted] (0,5)--(20,5);
\node at (5.5,5.5){ \textcircled{$\scriptstyle{1}$}};
\node at (6.5,4.5){ \textcircled{$\scriptstyle{2}$}};
\node at (7.5,3.5){ \textcircled{$\scriptstyle{3}$}};
\node at (8.5,2.5){ \textcircled{$\scriptstyle{4}$}};
\node at (0.5,4.5){ \textcircled{$\scriptstyle{5}$}};
\node at (15.5,0.5){ \textcircled{$\scriptstyle{6}$}};
\node at (3.5,2.5){ \textcircled{$\scriptstyle{7}$}};
\node at (11.5,0.5){ \textcircled{$\scriptstyle{8}$}};
\node at (4.5,1.5){ \textcircled{$\scriptstyle{9}$}};
\node at (0.5,2.5){ \textcircled{$\scriptstyle{1\hspace{-0.5mm}0}$}};
\node at (6.5,0.5){ \textcircled{$\scriptstyle{1\hspace{-0.5mm}1}$}};
\node at (0.5,1.5){ \textcircled{$\scriptstyle{1\hspace{-0.5mm}2}$}};
\node at (3.5,0.5){ \textcircled{$\scriptstyle{1\hspace{-0.5mm}3}$}};
\node at (1.5,0.5){ \textcircled{$\scriptstyle{1\hspace{-0.5mm}4}$}};
\node at (0.5,0.5){ \textcircled{$\scriptstyle{1\hspace{-0.5mm}5}$}};
%
%
\end{tikzpicture}
}
\\
\end{array}%
\end{align*}
\\


\noindent({\tt Step} 8) Set \(\lambda' = \lambda_{\not \succeq \bv}\). If \({\tt flip} = 1\), set \(\bv := \bv^T\) and \(\lambda' := (\lambda')^T\). Redefine \(\bu\) to be the concatenation \(\bu\bv\). Go to ({\tt Step} 9).\\

\noindent({\tt Step} 9) If \(\lambda' = \varnothing\), END and return \(\bu\). Otherwise, redefine \(\lambda:= \lambda'\) and go to ({\tt Step} 1). 
\end{Algorithm}

\begin{proof}
First, one must check that the algorithm is well-defined; this entails verifying that the diagrams depicted in ({\tt Steps} 2--7) describe a valid element sequence \(\bv\) (in particular, that the row labels are non-negative integers), and rests on the modular conditions for \(t\) below each diagram. This is a straightforward exercise, and is left to the reader.

To begin, we consider the case \(\lambda = \llbracket 6,6 \rrbracket\). As discussed in Example~\ref{66ex}, an exhaustive check shows that \(\M_2(\lambda, \bw) > 8\) for all \(\lambda\)-strategies \(\bw\), so \(\m_2(\lambda) > 8\). The \(\lambda\)-strategy defined in ({\tt Step} 7) of the algorithm yields \(8 < \m_2(\lambda) \leq \M_2(\lambda, \bu) =9\), so we have \(\m_2(\lambda) = \M_2(\lambda,\bu) = 9\), as desired. 


With that special case out of the way, we now prove, for all other diagrams under consideration, that Algorithm~\ref{Alg} produces a \(\lambda\)-strategy \(\bu\) such that \(\m_2(\lambda) = \M_2(\lambda, \bu) = \tau_2(|\lambda|)\). We go by induction on \(|\lambda|\). 
The base case \(\lambda = \llbracket 1,1 \rrbracket\) is clear, as the algorithm produces \(\bu = ((1,1))\), and so \(\M_2(\lambda, \bu) = \m_2(\lambda) = 1\).

Now let \( \lambda = \llbracket m,n \rrbracket\), where \(m \leq 6\) or \(n \leq 6\), and \(m,n \) are not both 6. Make the induction assumption that, if \(\nu\) satisfies these conditions as well, with \(|\nu| < |\lambda|\), then Algorithm~\ref{Alg} produces a \(\nu\)-strategy \(\bw\) such that \(\m_2(\nu) = \M_2(\nu, \bw) = \tau_2(|\nu|)\).

Via the transpose operations in ({\tt Steps} 1,8), it is enough to consider the `horizontally-oriented' situation \(m \leq n\), so we make that additional assumption now. We insert \(\lambda = \llbracket m,n \rrbracket\) into Algorithm~\ref{Alg}, letting \(t = \tau_2(|\lambda|) = \tau_2(mn)\), and letting the element sequence \(\bv = (v_1, \ldots, v_s)\) be as it stands at the end of ({\tt Step} 8) in the first loop of the algorithm. We begin by arguing that \(\lambda, \bv\) satisfy the following three conditions:

\begin{enumerate}
\item[(C1)] \(\M_2(\lambda, \bv) \leq t\).
\item[(C2)] \(|\lambda_{\succeq \bv}| +T_2(t)  - |\lambda| \geq  st - T_2({s-1})\)
\item[(C3)] \(\lambda_{\not \succeq \bv} \neq \llbracket 6,6 \rrbracket\).
\end{enumerate}

First we check that (C1) is satisfied by considering every diagram in ({\tt Steps} 2--7), save for the \(\llbracket 6,6 \rrbracket\) diagram. The homogeneously-colored component of the diagram marked with the element \({ \textcircled{$\scriptstyle{i}$}}\) in the southwest corner is exactly the set \(\lambda_{\bv}^{(i)}\). By adding up the elements in each row of \(\lambda_{\bv}^{(i)}\), it is straightforward to check that in all cases, we have  \(|\lambda_{\bv}^{(i)}| \leq t - i + 1\). Then we have:
\begin{align*}
\M_2(\lambda, \bv) = \max\{ |\lambda_{\bv}^{(i)}| + i - 1 \mid i = 1, \ldots, s\}
\leq \max\{ ( t - i + 1) + i -1 \mid i = 1, \ldots, s\}
=
t.
\end{align*}

Now we check that \(\lambda, \bv\) satisfy (C2) by considering every diagram in ({\tt Steps} 2--7), save for the \(\llbracket 6,6 \rrbracket\) diagram. We do so in the separate Cases 1--7 below.

(Case 1) Consider the small cases of the form:
\begin{itemize}
\item \(m=3\), \(n \in \{4, 5\}\) (and so \(t = 5\))
\item \(m=4\), \(n \in \{4,\ldots, 7\}\) (and so \(t  \in \{6,7\}\))
\item \(m=5\), \(n \in \{6,7\}\) (and so \(t = 8\))
\item \(m=6\), \(n \in \{7, \ldots, 11\}\) (and so \(t  \in\{ 9, \ldots, 11\}\))
\end{itemize}
In all these cases, we have \(s = 1\), and \(|\lambda_{\succeq \bv}| = n\). It is easily checked on a case-by-case basis that \(T_2(t) - |\lambda| = T_2(t) - mn \geq t-n\), so we have
\begin{align*}
|\lambda_{\succeq \bv}| + T_2(t) - |\lambda| 
\geq
n + (t-n)
=
t
=
 t - T_2(0),
\end{align*}
satisfying (C2).

(Case 2) Consider the small cases of the form:
\begin{itemize}
\item \(m=6\), \(n \in \{16, 17, 18, 19, 23, 24\}\) (and so \(t  \in\{14, 15, 17\}\))
\end{itemize}
In all these cases, we have \(s = 1\), and \(|\lambda_{\geq \bv}| = 6 \cdot \lfloor \frac{t}{6}\rfloor = 12\). It is easily checked on a case-by-case basis that \(12 + T_2(t) - 6n
\geq 
t\), so we have
\begin{align*}
|\lambda_{\succeq \bv}| + T_2(t) - |\lambda| 
=
12 + T_2(t) - 6n
\geq 
t
= t - T_2(0),
\end{align*}
satisfying (C2).

(Case 3) Consider the case \(\lambda =\llbracket m,n \rrbracket\), where \(3 \leq m \leq 6\), and \( t \equiv 1 \pmod m\), as in ({\tt Steps} 4,5,6,7). Then \(s=1\), and \(T_2(t) \not \equiv 0 \pmod m\) by Lemma~\ref{BigTLem}. We also have \(|\lambda| \equiv 0 \pmod m\), so \(T_2({t}) - |\lambda| \geq 1\) follows by Lemma~\ref{DiffLem}. Therefore
\begin{align*}
|\lambda_{\succeq \bv}| + T_2(t) - |\lambda| 
\geq
|\lambda_{\succeq \bv}| + 1
=
m \left\lfloor \frac{t}{m} \right\rfloor +1 = m\cdot  \frac{t-1}{m} +1 = t
=t - T_2(0),
\end{align*}
satisfying (C2).

(Case 4) Consider the case \(\lambda = \llbracket 5,n \rrbracket\) and  \(t \equiv 2 \pmod 5\), as in ({\tt Step} 6). Then \(s=1\), and \(T_2(t)  \equiv T_2(2) \equiv 3 \pmod 5\) by Lemma~\ref{BigTLem}. We also have \(|\lambda| \equiv 0 \pmod 5\), so \(T_2({t}) - |\lambda| \geq 3\) follows by Lemma~\ref{DiffLem}. Therefore
\begin{align*}
|\lambda_{\succeq \bv}| + T_2(t) - |\lambda| 
\geq
|\lambda_{\succeq \bv}| + 3
=
5 \left\lfloor \frac{t}{5} \right\rfloor +3 = 5\cdot  \frac{t-2}{5} +3 = 
t+1
>
t
=t - T_2(0),
\end{align*}
satisfying (C2).

(Case 5) Consider the case \(\lambda = \llbracket 5,n \rrbracket\) and \(t \equiv 3 \pmod 5\). Then we have \(t \equiv 3 \pmod{10}\) or \(t \equiv 8 \pmod{10}\) as in  ({\tt Step} 6). Then in either case \(s=3\), and \(T_2(t)  \equiv T_2(3) \equiv 1 \pmod 5\) by Lemma~\ref{BigTLem}. We also have \(|\lambda| \equiv 0 \pmod 5\), so \(T_2({t}) - |\lambda| \geq 1\) follows by Lemma~\ref{DiffLem}.
Therefore
\begin{align*}
|\lambda_{\succeq \bv}| + T_2(t) - |\lambda| 
\geq
|\lambda_{\succeq \bv}| + 1
=
(3t-4) + 1= 
3t- 3
=3 t - T_2(2),
\end{align*}
satisfying (C2).

(Case 6) Consider the case \(\lambda = \llbracket 6,n \rrbracket\) and \(16 \leq t \equiv 4 \pmod 6\), as in ({\tt Step} 7). Then \(s=2\), and \(T_2(t) \not \equiv 0 \pmod 6\) by Lemma~\ref{BigTLem}. We also have \(|\lambda| \equiv 0 \pmod 6\), so \(T_2({t}) - |\lambda| \geq 1\) follows by Lemma~\ref{DiffLem}. Therefore 
\begin{align*}
|\lambda_{\succeq \bv}| + T_2(t) - |\lambda| 
\geq 
|\lambda_{\succeq \bv}| + 1
=
(2t-2) + 1
= 2t-1
= 2t - T_2(1),
\end{align*}
satisfying (C2).

(Case 7) Now we may consider the remaining cases in one fell swoop. In all remaining cases, it may be checked that \(\bv\) is defined such that \(|\lambda_{\bv}^{(i)}| = t- i + 1\) for \(t=1, \ldots, s\), and thus \(|\lambda_{\succeq \bv}| = st - T_2({s-1})\). Therefore we have
\begin{align*}
|\lambda_{\succeq \bv}| + T_2(t) - |\lambda| \geq |\lambda_{\succeq \bv}| = st - T_2({s-1}),
\end{align*}
satisfying (C2).

Now we check that \(\lambda,\bv\) satisfy (C3). The case \(\lambda =\llbracket m,n \rrbracket\) for \(m < 6\) is obvious. Thus we may assume that \(\lambda = \llbracket 6,n \rrbracket\). As with the last claim, we check (C3) in the separate Cases 1--10 below.

(Case 1) If \(7 \leq n \leq 11\), then \(\lambda_{\not \succeq \bv}\) is a 5-row diagram, so is not equal to \(\llbracket 6,6 \rrbracket\).

(Case 2) If \(12 \leq n \leq 15\), then \(t \in \{12, 13\}\), so \(t \equiv 0, 1 \pmod 6\). Then \(\lfloor \frac{t}{6} \rfloor  = 2\), and we have
\(
|\lambda_{\not \succeq \bv}| = |\lambda| - 6 \cdot \left \lfloor \frac{t}{6} \right \rfloor \geq 12\cdot 6 - 12 = 60,
\)
so \(\lambda_{\not \succeq \bv} \neq \llbracket 6,6 \rrbracket\). 

(Case 3) If \(n \in \{16, 17, 18, 19, 23, 24\}\), then \(t \in \{14, 15, 17\}\), so \(\lfloor \frac{t}{6} \rfloor  = 2\), and we have
\(
|\lambda_{\not \succeq \bv}| = |\lambda| - 6 \cdot \left \lfloor \frac{t}{6} \right \rfloor \geq 16\cdot 6 - 12 = 84,
\)
so \(\lambda_{\not \succeq \bv} \neq \llbracket 6,6 \rrbracket\). 

(Case 4) If \(n = 20\), then \(\lambda_{\not \succeq \bv} = \varnothing \neq \llbracket 6,6 \rrbracket\).

(Case 5) If \(n \in \{21, 22\}\), then \(t = 16\). Then \(|\lambda_{\not \succeq \bv}| = |\lambda| - (2t -2) \geq (21 \cdot 6) - 30 = 96\), so \(\lambda_{\not \succeq \bv} \neq \llbracket 6,6 \rrbracket\).

(Case 6) If \(n=25\), then \(t = 17\). Then \(|\lambda_{\not \succeq \bv}| = |\lambda| - (8t - 28) = (25 \cdot 6) - (8 \cdot 17 - 28) = 42\), so \(\lambda_{\not \succeq \bv} \neq \llbracket 6,6 \rrbracket\).

In the remaining cases, we assume that \(n \geq 26\). Then we have \(t \geq 18\). Note that by the definition of \(t = \tau_2(|\lambda|)\), we have \(|\lambda| > T_2({t-1})\).

(Case 7) If \(t \equiv 0, 1 \pmod 6\), then
\begin{align*}
|\lambda_{\not \succeq \bv}| &= |\lambda| - 6 \left\lfloor \frac{t}{6} \right\rfloor > T_2({t-1}) - t = T_2({t-1}) - (t-1) -1\\
&= T_2({t-2}) -1 \geq T_2({16}) -1 = 135,
\end{align*}
so \(\lambda_{\not \succeq \bv} \neq \llbracket 6,6 \rrbracket\).

(Case 8) Say \(t \equiv 4 \pmod 6\). Then
\begin{align*}
|\lambda_{\not \succeq \bv}| &= |\lambda| - (2t-2) > T_2({t-1}) - (2t-2)\\
&= T_2({t-1}) - (t-1) - (t-2) -1 = T_2({t-3}) - 1 \geq T_2({15}) - 1 = 119,
\end{align*}
so \(\lambda_{\not \succeq \bv} \neq \llbracket 6,6 \rrbracket\).

(Case 9) Say \(t \equiv 3 \pmod 6\). Then 
\begin{align*}
|\lambda_{\not \succeq \bv}| &= |\lambda| - (7t - 21) > T_2({t-1}) - (t-1) - \cdots  - (t-7) - 7\\
&= T_2({t-8}) - 7 \geq T_2({10}) - 7 = 48, 
\end{align*}
so \(\lambda_{\not \succeq \bv} \neq \llbracket 6,6 \rrbracket\).

(Case 10) Say \(t \equiv 2,5 \pmod 6\). Then
\begin{align*}
|\lambda_{\not \succeq \bv}| &= |\lambda| - (8t - 28) > T_2({t-1}) - (t-1) - \cdots - (t-8) - 8\\
&= T_2({t-9}) - 8 \geq T_2(9) - 8  = 37,
\end{align*}
so \(\lambda_{\not \succeq \bv} \neq \llbracket 6,6 \rrbracket\).

Thus, in every case we have \(\lambda_{\not \succeq \bv} \neq \llbracket 6,6 \rrbracket\), and so (C3) holds. 

Therefore (C1), (C2), (C3) hold for \(\lambda, \bv\). By (C2), we have 
\begin{align*}
|\lambda_{\not \succeq \bv}| = |\lambda| - |\lambda_{\succeq \bv}|
\leq
T_2({\tau_2(|\lambda|)}) - s \tau_2(|\lambda|) + T_2({s-1}),
\end{align*}
so by Lemmas~\ref{increasing} and \ref{TauLemma} we have
\begin{align}\label{eqn5}
\tau_2(|\lambda_{\not \succeq \bv}|)
\leq 
\tau_2(T_2({\tau_2(|\lambda|)}) - s \tau_2(|\lambda|) + T_2({s-1}))
\leq
 \tau_2(|\lambda|) - s.
\end{align}
By (C3), the induction assumption holds for \(\lambda_{\not \succeq \bv}\), so inserting \(\lambda_{\not \succeq \bv}\) into the algorithm yields a \(\lambda_{\not \succeq \bv}\)-strategy \(\bw\) such that \(\M_2(\lambda_{\not \succeq \bv}, \bw) = \m_2(\lambda_{\not \succeq \bv}) = \tau_2(|\lambda_{\not \succeq \bv}|)\). By the inductive nature of the algorithm, inserting \(\lambda\) into the algorithm yields the \(\lambda\)-strategy \(\bu= \bv\bw\). Then we have
\begin{align*}
\M_2(\lambda, \bu) &= \max\{ \M_2(\lambda, \bv), \M_2(\lambda_{\not \succeq \bv}, \bw)+s\} & \textup{by Lemma~\ref{truncatelem}}\\
& \leq  \max\{ \tau_2(|\lambda|), \M_2(\lambda_{\not \succeq \bv}, \bw)+s\} & \textup{by (C1)}\\
&= \max\{ \tau_2(|\lambda|), \tau_2(|\lambda_{\not \succeq \bv}|) + s\} & \textup{by induction assumption}\\
&\leq \max\{ \tau_2(|\lambda|), (\tau_2(|\lambda|) - s) + s\} & \textup{by (\ref{eqn5})}\\
&= \tau_2(|\lambda|).
\end{align*}
Then, by Theorems~\ref{BOUNDS} and~\ref{sameq}, we have
\begin{align*}
\tau_2(|\lambda|) \leq \m_2(\lambda) \leq \M_2(\lambda, \bu) \leq \tau_2(|\lambda|),
\end{align*}
so we have \(\m_2(\lambda) = \M_2(\lambda, \bu) = \tau_2(|\lambda|)\) as desired, completing the proof.
\end{proof}

As \(\kappa \times \nu \cong \llbracket m, n \rrbracket\) when \(\kappa\), \(\nu\) are totally ordered sets of cardinality \(m,n\) respectively, we have the immediate corollary thanks to Lemma~\ref{isomq}:

\begin{Corollary}\label{maincor}
Let \(\kappa, \nu\) be finite totally ordered sets, with \(|\kappa| \leq 6\) or \(|\nu| \leq 6\). Then
\begin{align*}
\textup{\(\m_2(\kappa \times \nu)\)} = \begin{cases}
9 & \textup{if }|\kappa|=|\nu| =6;\\
\tau_2(|\kappa| |\nu|) & \textup{otherwise}.
\end{cases}
\end{align*}
\end{Corollary}

\subsection{A conjecture}\label{conjsec}
We end with a conjectural bound for product posets of totally ordered sets.

\begin{Conjecture}\label{conj}
Let \(m,n \in \NN\). Then
\(
\m_2(\llbracket m,n \rrbracket) \leq \tau_2(mn) + 1.
\)
\end{Conjecture}

This suggests \(\m_2(\llbracket m,n \rrbracket) \in \{\tau_2(mn), \tau_2(mn) + 1\}\) for all \(m,n \in \NN\). 
By Theorem~\ref{mainthm}, the posets \(\llbracket m,n \rrbracket\) obey this claim when \(m \leq 6\) or \(n \leq 6\). In fact, all but \(\llbracket 6,6 \rrbracket\) have the minimal possible value \(\m_2(\llbracket m,n \rrbracket)=\tau_2(mn)\) allowed by Theorem~\ref{BOUNDS}. 
Moving beyond these results, computations show that exceptional cases like \(\llbracket 6,6 \rrbracket\), where no \(\lambda\)-strategy \(\bu\) can be found that realizes \(\M_2(\lambda, \bu) = \tau_2(|\lambda|)\), seem to occur fairly rarely (the poset \(\llbracket 15,20 \rrbracket\) is another). But allowing for a \(\lambda\)-strategy that realizes \(\M_2(\lambda, \bu) = \tau_2(|\lambda|) + 1\) instead seems to afford so much flexibility that we expect such a \(\lambda\)-strategy can always be found, even in these exceptional cases. For instance, while there are no \(\llbracket 6,6 \rrbracket\)-strategies that realize \(\M_2(\llbracket 6,6 \rrbracket, \bu) = 8\), there are 53,688 distinct \(\llbracket 6,6 \rrbracket\)-strategies which realize \(\M_2(\llbracket 6,6\rrbracket, \bu) = 9\). This is the authors' line of reasoning behind positing Conjecture~\ref{conj}.


\begin{thebibliography}{9} 
    
    

        \bibitem{Davey}
B. A. Davey, H. A. Priestley, {\em Introduction to Lattices and Order}, Cambridge University Press, New York, 2002.    

    \bibitem{marbles}
    R. Denman, D. Hailey, and M. Rothenberg,
    The Tower and Glass Marbles Problem,
    {\em College Math. J.}, {\bf 41} 5 (2010) 350--356.   
    
    
    \bibitem{Dushnik}
    B. Dushnik and E. Miller, Partially ordered sets, {\em Amer. J. Math.}, {\bf 63} (1941) 600--610.
    
      
    \bibitem{Egg1}
J. D. E. Konhauser, D. Velleman, and S. Wagon, (1996). {\em Which way did the Bicycle Go?} Dolciani Mathematical Expositions, No. 18, The Mathematical Association of America, 1996.

   \bibitem{Egg2}
   G. L. McDowell, {\em Cracking the Coding Interview, 5th ed.}, CareerCup, 2011.


  \bibitem{Egg3}
    S.S. Skiena, {\em The Algorithm Design Manual}, Springer-Verlag, New York, 1997.


      
\end{thebibliography}
\end{document}